\documentclass[reqno]{amsart}

\usepackage[utf8]{inputenc}
\usepackage{amsmath}
\numberwithin{equation}{section}
\numberwithin{figure}{section}
\numberwithin{table}{section}
\usepackage{graphicx}
\usepackage{amssymb}
\usepackage{mathrsfs}
\usepackage{amsfonts}
\usepackage[colorlinks=true, allcolors=blue]{hyperref}
\usepackage{subfigure} 
\usepackage{color}
\usepackage{extarrows}
\usepackage{booktabs}
\usepackage{hyperref}
\usepackage{tikz}
\usepackage{extarrows}
\usepackage{booktabs}
\usepackage{amsthm}
\usepackage{enumerate}
\usepackage{enumitem} 
\usepackage{indentfirst}
\usepackage{embrac}
\usepackage{cite} 
\usetikzlibrary{positioning,shapes}

\def\bZ{{\mathbb Z}}

\def\bF{{\mathbf F}}

\def\bR{{\mathbb R}}

\def\sE{{\mathscr E}}
\def\sF{{\mathscr F}}

\def\sG{{\mathscr G}}

\def\bs{\mathbf{s}}
\def\br{\mathbf{r}}
\def\bN{\mathbb{N}}
\def\sS{\mathscr{S}}

\def\fm{\mathfrak{m}}

\def\${|\!|\!|}
\def\l|{\left|\!\left|\!\left|}
\def\r|{\right|\!\right|\!\right|}

\newtheorem{theorem}{Theorem}[section]
\newtheorem{lemma}[theorem]{Lemma}
\newtheorem{proposition}[theorem]{Proposition}
\newtheorem{corollary}[theorem]{Corollary}
\theoremstyle{definition}
\newtheorem{definition}[theorem]{Definition}
\newtheorem{example}[theorem]{Example}
\newtheorem{hypothesis}[theorem]{Hypothesis}

\theoremstyle{remark}
\newtheorem{remark}[theorem]{Remark}
\numberwithin{equation}{section}

\begin{document}

\title[Diffusions with discontinuous scale]{Dirichlet form approach to diffusions with discontinuous scale}

\author{Liping Li}
\address{Fudan University, Shanghai, China.  }
\address{Bielefeld University,  Bielefeld, Germany.}
%\curraddr{Bielefeld University,  Bielefeld, Germany.}
\email{liliping@fudan.edu.cn}
\thanks{The author is a member of LMNS,  Fudan University.  He is also partially supported by NSFC (No.  11931004) and Alexander von Humboldt Foundation in Germany.  }

%\author{Hanlai Lin}
%\address{Fudan University, Shanghai, China.  }
%\email{hanlailin18@163.com}

%    General info

%\date{January 1, 2001 and, in revised form, June 22, 2001.}



\begin{abstract}
It is well known that a regular diffusion on an interval $I$ without killing inside is uniquely determined by a canonical scale function $\bs$ and a canonical speed measure $\fm$.  Note that $\bs$ is a strictly increasing and continuous function and $\fm$ is a fully supported Radon measure on $I$.  In this paper we will associate a general triple $(I,\bs,\fm)$,  where $\bs$ is only assumed to be increasing and $\fm$ is not necessarily fully supported,  to certain Markov processes by way of Dirichlet forms.  A straightforward generalization of Dirichlet form associated to regular diffusion will be first put forward,  and we will find out its corresponding continuous Markov process $\dot X$,  for which the strong Markov property fails whenever $\bs$ is not continuous.  Then by operating regular representations on Dirichlet form and Ray-Knight compactification on $\dot X$ respectively,  the same unique desirable symmetric Hunt process associated to $(I,\bs,\fm)$ is eventually obtained.  This Hunt process is homeomorphic to a quasidiffusion,  which is known as a celebrated generalization of regular diffusion.
%Using two transformations,  called scale completion and darning respectively,  to rebuild the topology of $I$,  we will successfully regularize the triple $(I,\bs,\fm)$ and obtain a regular Dirichlet form associated to it.  The corresponding Markov process is called the regularized Markov process associated to $(I,\bs,\fm)$.  In fact,  it is the unique Markov process up to homeomorphism that can be associated to $(I,\bs,\fm)$ in the context of regular representations of Dirichlet forms.  As a byproduct of regularized Markov process,  a continuous simple Markov process,  which does not satisfy the strong Markov property,  will be also raised to be associated to $(I,\bs,\fm)$ without operating regularizing program.  Furthermore,  we will show that the regularized Markov process is  identified with a skip-free Hunt process in one dimension as well as a quasidiffusion without killing inside.  Note that the skip-free Hunt process generalizes the concept of regular diffusion and admits a scale function and a speed measure in an analogous manner. 

% In this talk we will introduce the Dirichlet forms related to a triple $(I,s,m)$,  where $I$ is a closed interval,  $s$ is a (not necessarily strictly) increasing function on $I$, and $m$ is a certain positive measure on $I$.    These Dirichlet forms are associated to so-called skip-free Hunt processes on an almost closed set in one dimension,  which generalize classical regular diffusions.  We will also show the correspondence between Dirichlet forms,  skip-free Hunt processes and quasidiffusions.  
\end{abstract}

\maketitle
\tableofcontents

\section{Introduction}

The title is a little misleading because the scale function of a nice diffusion in one dimension is always continuous.  Actually it focuses on the problem that how to associate an increasing but not necessarily continuous function on an interval to a certain Markov process.  But it is roughly correct because the final result shows that the desirable Markov process is very similar to a diffusion,  besides that the continuity of sample paths is replaced by the so-called \emph{skip-free property}.  This problem has been considered by Sch\"utze \cite{S79} by way of generalizing the \emph{second order differential operator} raised by Feller,  and it turns out that the resulting process thereof is closely related to a widely studied Markov process,  called \emph{quasidiffusion} in,  e.g.,  \cite{BK87,  K86},  \emph{generalized diffusion} in,  e.g.,  \cite{W74,  KW82, LM20} and \emph{gap diffusion} in,  e.g.,  \cite{K81}.  In this paper we will adopt another treatment by virtue of the theory of Dirichlet forms.  A Dirichlet form is a closed symmetric form with Markovian property  on an $L^2$-space.  Due to a series of important works by Fukushima and Silverstein in the 1970s,  the ``\emph{regularity}" of a Dirichlet form assures that it is associated to a symmetric \emph{Hunt process}.  We refer readers to \cite{FOT11, CF12} for notations and terminologies in the theory of Dirichlet forms. 

What is a diffusion in one dimension? As one of the most important stochastic models,  it is a continuous strong Markov process $X=(X_t)_{t\geq 0}$ on an interval $I=\langle l, r\rangle$ where $l$ or $r$ may or may not be contained in $I$; see,  e.g., \cite{I06,  IM74,  M68}.  Due to the study \cite{IM74} by It\^o and McKean,  every diffusion can be decomposed into ``regular" pieces: Like the concept of irreducibility for Markov chains,  this regularity means that every point in a piece can be visited by the diffusion starting from any other point in the same piece in finite time.  So one loses little generality and gains much simplification by considering only a regular diffusion.  More precisely,  $X$ is called \emph{regular} if $\mathbf{P}_x(T_y<\infty)>0$ for any $x\in \mathring{I}:=(l,r)$ and $y\in I$,  where $T_y:=\inf\{t>0: X_t=y\}$.  For simplification we further assume that $X$ has no killing inside in the sense that $X_{\zeta-}\notin I$ if $\zeta<\infty$ where $\zeta$ is the lifetime of $X$.  Then a significant characterization tells us that $X$ is uniquely determined by a \emph{canonical scale function} $\bs$ and a \emph{canonical speed measure} $\fm$; see,  e.g.,  \cite[V\S7]{RW87} and \cite[VII\S3]{RY99}.  Note that $\bs$ is a continuous and strictly increasing function on $I$ and $\fm$ is a fully supported Radon measure on $I$.  

To our knowledge,  there appeared at least two important analytic treatments to reach a regular diffusion.  (So we have at least two possible methods to generalize it.) 
The first one brings into play the generalized second order differential operator 
\begin{equation}\label{eq:01}
\mathscr L:=\frac{1}{2}\frac{d^2}{d\fm d\bs}
\end{equation}
and its associated \emph{Feller semigroup}.  This semigroup leads to a \emph{Feller process},  which is identified with the expecting diffusion.   A systematic introduction is referred to in \cite{M68}.  Another way is to make use of Dirichlet forms.  As far as we know,  Fang et al.  \cite{FHY10} first put forward the regular Dirichlet form on $L^2(I,\fm)$ associated to this diffusion:
\begin{equation}\label{eq:13-2}
\begin{aligned}
	&\sF^{(\bs,\fm)}=\{f\in L^2(I,\fm):f\ll \bs,  df/d\bs\in L^2(I,d\bs),  \\
	&\qquad\qquad \qquad f(j)=0\text{ if } j\notin I\text{ and }|\bs(j)|<\infty  \text{ for }j=l\text{ or }r\}, \\
	&\sE^{(\bs,\fm)}(f,g)=\frac{1}{2}\int_{I}\frac{df}{d\bs}\frac{dg}{d\bs}d\bs,\quad f,g\in \sF^{(\bs,\fm)},
\end{aligned}
\end{equation}
where $f\ll \bs$ stands for that $f$ is absolutely continuous with respect to $\bs$; see also \cite{F10}.  With the help of theory of Dirichlet forms,  one can go farther in related studies.  For example,  the correspondence between the Dirichlet form \eqref{eq:13-2} and the operator \eqref{eq:01} was studied in \cite{F14}.  Dirichlet form characterization for diffusions without regular property was accomplished in \cite{LY19,  L21}.  %Rich ``singular" diffusions were found out in a series of articles concerning so-called \emph{regular Dirichlet subspaces}; see,  e.g,  \cite{LY17,  LY19-2,  LSY20}.  

% giving the hitting distributions of $X$ in the sense that for any $a,x,b\in I$ with $a<x<b$,
%\begin{equation}\label{eq:02}
%	\mathbf{P}_x(T_b<T_a)=\frac{\bs(x)-\bs(a)}{\bs(b)-\bs(a)}. 
%\end{equation}
%When $\bs(x)=x$,  $X$ is called \emph{on its natural scale}. 
%The canonical speed measure $\fm$ is a fully supported Radon measure on $I$,  which is roughly defined as $-\frac{1}{2}h''_{a,b}$ on every open interval $(a,b)\subset I$,  where $h_{a,b}(x):=\mathbf{E}_xT_a\wedge T_b$ is concave and $-h''_{a,b}$ is the Radon measure induced by the second derivative of $-h_{a,b}$ in the sense of distribution.  A more comprehensible explanation for canonical speed measure is as follows. When on its natural scale,  $X$ can be expressed as a time change of Brownian motion and $\fm$ measures the speed of its movements: In regions where $\fm$ is large,  $X$ moves slowly.  Particularly,  the Brownian motion is on its natural scale and its canonical speed measure is the Lebesgue measure.  

To generalize the operator \eqref{eq:01},  Kac and Krein \cite{KK74} initialized a spectral theory,  known as \emph{Krein's correspondence}, for the case that $\bs(x)=x$ and $\fm$ is determined by a right continuous and (not strictly) increasing function.  This theory applied to Markov processes and led to celebrated quasidiffusions; see,  e.g.,  \cite{K75,  K81,  KW82}.  At a heuristic level,  a quasidiffusion may be thought of as the \emph{trace} of Brownian motion on the topological support of $\fm$.  Then Sch\"utze \cite{S79} studied the operator \eqref{eq:01} for another case that $\bs$ is only strictly increasing and $\fm$ is a fully supported Radon measure such that $\fm$ has an isolated mass at points where $\bs$ is neither right nor left continuous.  It is insightful to point out in \cite{S79} that \eqref{eq:01} still corresponds to a ``simple" Markov process on $I$,  whereas the strong Markov property may fail;  see \cite[Corollary~4.7]{S79}.  To associate \eqref{eq:01} to a strong Markov process, Sch\"utze operated the completion of state space $I$ with respect to $\bs$,  so that a desirable Feller semigroup was obtained on the completed space.

In the current paper we will develop a second way to generalize regular diffusions.  
The object is to put forward a generalization of \eqref{eq:13-2} and to associate it to certain Markov process by virtue of the theory of Dirichlet forms.  More precisely,  an increasing function $\bs$ on $I$ and a positive Radon measure $\fm$ on $I$ will be considered.  Note that $\bs$ is not assumed to be strictly increasing nor left/right continuous,  and $\fm$ is not necessarily fully supported.  Clearly this pair $(\bs,\fm)$ is more general than that in \cite{S79}.  By explaining the differential term $df/d\bs(x)dg/d\bs(x)d\bs(x)$ at $x\in D^\pm:=\{x\in I:\bs(x)\neq \bs(x\pm)\}$ as $(f(x)-f(x\pm))(g(x)-g(x\pm))/(\bs(x)-\bs(x\pm))$ for suitable functions $f$ and $g$,  we raise a quadratic form $(\sE,\sF)$ on $L^2(I,\fm)$, explicitly expressed as \eqref{eq:25},  in terms of $(\bs,\fm)$.  As coincides with \eqref{eq:13-2} whenever $\bs$ is strictly increasing and continuous and $\fm$ is fully supported,  it is a reasonable generalization of \eqref{eq:13-2}.  But the challenge is that $(\sE,\sF)$ is not expected to be a regular Dirichlet form.  We can only prove in Theorem~\ref{LM12} that it is a Dirchlet form \emph{in the wide sense},  i.e.  a non-negative symmetric closed form satisfying the Markovian property (see \cite[\S1.4]{FOT11} for these terminologies).  Nevertheless,  inspired by the remarkable observation in \cite[Corollary~4.7]{S79},  we eventually obtains in Theorem~\ref{THM82} an $\fm$-symmetric continuous Markov process $\dot X$ on $I$,  for which the strong Markov property fails whenever $\bs$ is not continuous.  The Dirichlet form of $\dot X$ on $L^2(I,\fm)$ is exactly $(\sE,\sF)$. 

It is then interesting to find out certain methods to ``regularize" $(\sE,\sF)$ or $\dot X$ to obtain a nicer Markov process.  Regarding $(\sE,\sF)$,  we will adopt a standard argument,  known as \emph{regular representation},  which was raised by Fukushima in a seminal paper \cite{F71}.  By applying this argument to every Dirichlet form in the wide sense,  a family of regular Dirichlet forms will  be obtained, and all of them are \emph{quasi-homeomorphic} in the sense of,  e.g.,  \cite[Definition~1.4.1]{CF12}.  In our case we will prove further in Theorem~\ref{THM311} that all regular representations of $(\sE,\sF)$ are not only quasi-homeomorphic but also homeomorphic.  Hence up to homeomorphisms,  $(\sE,\sF)$ is associated to a unique Hunt process by way of regular representation.  We call it the \emph{regularized Markov process associated to} the triple $(I,\bs,\fm)$.  On the other hand,  the ``regularizing" operation of $\dot X$ is rather standard as well.  
That is the celebrated \emph{Ray-Knight compactification}.  The Ray-Knight compactification,  raised by Ray \cite{R59} and firmly established after a revision by Knight \cite{K65},  provides a wonderful way to modify an (almost arbitrary) Markov process very slightly and to result in a nice Markov process on a merely larger topology space,  called \emph{Ray process}.  The original process is assumed to enjoy neither strong Markov property nor certain regularity of sample paths.  The resulting Ray process is,  however,  c\`adl\`ag and strong Markovian.  A brief review of Ray processes and Ray-Knight compactification will be shown in Appendix~\ref{APPB}.  Returning back to our investigation,  we can arrive at a Ray process by operating Ray-Knight compactification on $\dot X$.  As is expected,  this Ray process is identified with the regularized Markov process associated to $(I,\bs,\fm)$.  In a word,  the methods of regular representation in analysis and Ray-Knight compactifiction in probability lead to the same desirable ``regularization" of $\dot X$.  

We should point out that among all regular representations of $(\sE,\sF)$,  a canonical one singled out in Theorem~\ref{THM6} is associated to a quasidiffusion.  In other words,  the regularized Markov process associated to $(I,\bs,\fm)$ is actually a quasidiffusion up to certain homeomorphism.

Throughout this paper a basic hypothesis (DK),  as stated in Hypothesis~\ref{HYP22},   will be assumed.  It bars the possibility of killing inside for $\dot X$ and its regularized Markov process.  In addition,  this hypothesis corresponds to another condition (QK) for quasidiffusions appearing in \cite{L23b}.  We wish to state emphatically that without assuming (DK),  the formulation of regularized Markov process still holds true,  while there may appear killing at endpoints contained in the state space; see Remarks~\ref{RM38}. 

% the condition (DK) for $(I,\bs,\fm)$,  the condition (SK) for a skip-free Hunt process and the condition (QK) for a quasidiffusion are assumed to bar the possibility of killing inside.   %The case of skip-free Hunt process is more complicated.  Without (SK),  a skip-free Hunt process may admit killing everywhere,  so a \emph{killing measure} should be introduced to characterize the killing part of the process.  We hope to make it in a further contribution. 

The paper is organized as follows.  In \S\ref{SEC2}, we will introduce the quadratic form $(\sE,\sF)$,  a generalization of \eqref{eq:01},  and the main result, Theorem~\ref{LM12}, shows that it is a Dirichlet form in the wide sense.  The section \S\ref{SEC3} is devoted to presenting explicit expression of regular representations of $(\sE,\sF)$ and to proving that all of them are homeomorphic.   
In \S\ref{SEC31} the simple Markov process $\dot X$ associated to $(\sE,\sF)$ will be put forward.  It turns out in \S\ref{SEC43-2} that the regularized Markov process associated to $(I,\bs,\fm)$ is exactly the Ray-Knight compactifiction of $\dot X$.  The last section \S\ref{SEC8} gives several examples of Markov processes that can be obtained by regularizing a certain triple $(I,\bs,\fm)$.  

%Two additional conditions marked as (DK) and (DM) are assumed.  

\subsection*{Notations}
Let $\overline{\mathbb{R}}=[-\infty, \infty]$ be the extended real number system.  A set $E\subset \overline{\bR}$ is called a \emph{nearly closed subset} of $\overline{\bR}$ if $\overline E:= E\cup \{l,r\}$ is a closed subset of $\overline{\bR}$ where $l=\inf\{x: x\in E\}$ and $r=\sup\{x: x\in E\}$.  The point $l$ or $r$ is called the left or right endpoint of $E$.  Denote by $\overline{\mathscr K}$ the family of all nearly closed subsets of $\overline{\bR}$.  Set
\[
	\mathscr K:=\{E\in \overline{\mathscr K}: E\subset \bR\},
\]
and every $E\in \mathscr K$ is called a \emph{nearly closed subset} of $\bR$.  

Let $E$ be a locally compact separable metric space.  We denote by $C(E)$ the space of all real valued continuous functions on $E$.  In addition,  $C_c(E)$ is the subspace of $C(E)$ consisting of all continuous functions on $E$ with compact support and
\[
	C_\infty(E):=\{f\in C(E): \forall \varepsilon>0, \exists K\text{ compact},  |f(x)|<\varepsilon, \forall x\in E\setminus K\}.  
\]
The functions in $C_\infty(E)$ are said to be vanishing at infinity.  
Given an interval $J$,  $C_c^\infty(J)$ is the family of all smooth functions with compact support on $J$.  

Given a continuous and increasing function $\bs$ on an interval $J$ and $f\in C(J)$,  $f\ll \bs$ means that there exists an absolutely continuous function $g$ on $\bs(J):=\{\bs(x):x\in J\}$ such that $f=g\circ \bs$.  Meanwhile $df/d\bs:=g'\circ \bs$.  

The abbreviations CAF and PCAF stand for \emph{continuous additive functional} and \emph{positive continuous additive functional} respectively.  

\section{Dirichlet forms associated to discontinuous scales}\label{SEC2}

%This section is devoted to introduce Dirichlet forms in wide sense related to discontinuous scales on an interval. 

\subsection{Discontinuous scale on an interval}

Let $-\infty\leq l<0<r\leq \infty$ and we are given two ingredients throughout this paper:
\begin{itemize}
\item[(a)] A non-constant increasing real valued function $\bs$,  called \emph{scale function},  on $(l,r)$ such that $\bs(0-)=\bs(0)=\bs(0+)=0$. 
\item[(b)] A positive measure $\fm$,  called \emph{speed measure},  on $[l,r]$, which is Radon on $(l,r)$,  i.e.  $\fm([a,b])<\infty$ for any $[a,b]\subset (l,r)$,  and such that $ \fm(\{0\})=0$.  
\end{itemize}
The conditions  $\bs(0-)=\bs(0)=\bs(0+)=0$ and $ \fm(\{0\})=0$ lose no generality,  and for convenience we make the conventions $\bs(l-)=\bs(l):=\lim_{x\downarrow l}\bs(x)$ and $\bs(r+)=\bs(r):=\lim_{x\uparrow r}\bs(x)$. 
It is worth emphasising that $\bs(l),\bs(r), \fm(\{l\})$ or $\fm(\{r\})$  may be infinite. 

Two basic hypotheses on $\bs$ and $\fm$ will be assumed.  Before stating them we need to prepare some concepts and notations.  
The first class of notions concerns the structure of $[l,r]$ in terms of $\bs$ and $\fm$.  Let $\text{supp}[\fm]$ be the \emph{support} of $\fm$,  i.e.  the smallest closed subset of $[l,r]$ outside which $\fm$ vanishes. 
Regarding the scale function,  note that $\bs$  is not assumed to be continuous nor strictly increasing.  Instead set
\begin{equation}\label{eq:13}
	U:=\{x\in [l,r]: \bs\text{ is constant on }(x-\varepsilon,x+\varepsilon)\cap [l,r]\text{ for some }\varepsilon>0\},
\end{equation}
and we call $\text{supp}[\bs]:=[l,r]\setminus U$ the \emph{support} of $\bs$.  In addition,  $U \setminus \{l,r\}$ is clearly open and hence can be written as a disjoint union of open intervals 
\begin{equation}\label{eq:14}
U\setminus \{l,r\}=\bigcup_{n\geq 1}(c_n,d_n). 
\end{equation}
Put $D^\pm:=\{x\in [l,r]: \bs(x)\neq \bs(x\pm)\}$ and $D^0:=D^+\cap D^-,  D:=D^+\cup D^-$.
Every point $x\in D^0$ is called \emph{isolated} (with respect to $\bs$),  and the interval $(c_n,d_n)$ is called \emph{isolated} (with respect to $\bs$) if $c_n,d_n\in D\cup \{l,r\}$.  The second class of notions classifies the boundaries $l$ and $r$ in terms of the pair $(\bs,\fm)$.  They are rather important and we summarize them as a definition.  For convenience $\fm(r-)<\infty$ (resp.  $\fm(l+)<\infty$) stands for that $\fm((r-\varepsilon,r))<\infty$ (resp.  $\fm((l,l+\varepsilon))<\infty$) for some $\varepsilon>0$.  Otherwise we make the convention $\fm(r-)=\infty$ (resp. $\fm(l+)=\infty$).

\begin{definition}\label{DEF21}
\begin{itemize}
\item[(1)] The endpoint $r$ (resp.  $l$) is called \emph{approachable} if $\bs(r)<\infty$ (resp.  $\bs(l)>-\infty$).  
\item[(2)] An approachable endpoint $r$ (resp.  $l$) is called \emph{regular} if $\fm(r-)<\infty$ (resp.  $\fm(l+)<\infty$).  
\item[(3)] A regular endpoint $r$ (resp.  $l$) is called \emph{reflecting},  if $\fm(\{r\})<\infty$ (resp.  $\fm(\{l\})<\infty$).  Otherwise it is called \emph{absorbing}. 
\end{itemize}
\end{definition}
%\begin{remark}
%When $\text{supp}[\fm]=[l,r]$ and $\bs$ is continuous and strictly increasing,  the boundary points can be classified into another Feller's type: \emph{regular,  exit,  entrance} and \emph{natural}; see,  e.g.,  \cite{IM74}.  Currently we can not do this for general $\bs$ because 
%\end{remark}

Based on this definition,  we put an interval $I:=\langle l,  r\rangle$ ended by $l$ and $r$,
where $l\in I$ (resp.  $r\in I$) if and only if $l$ (resp. $r$) is reflecting.  The restriction of $\fm$ to $I$,  still denoted by $\fm$,  is clearly a Radon measure.  Now we have a position to present two additional conditions on the triple $(I,\bs,\fm)$.

\begin{hypothesis}\label{HYP22}
The following conditions are assumed for $(I,\bs,\fm)$:
\begin{itemize}
\item[(DK)] $l$ or $r$ is reflecting if it is the endpoint of an isolated interval in \eqref{eq:14}.  
\item[(DM)]  $\text{supp}[\bs]\subset \text{supp}[\fm]$,  $\fm(\{x\})>0$ for $x\in D^0$ and for isolated interval $(c_n,d_n)$ in \eqref{eq:14},  it holds that $\fm\left(J_{c_n}^{d_n} \right)>0$,
where $J_{c_n}^{d_n}$ is an interval ended by $c_n$ and $d_n$ and $c_n\in J_{c_n}^{d_n}$ (resp.  $d_n\in J_{c_n}^{d_n}$) whenever $c_n\notin D^+$ (resp.  $d_n\notin D^-$).  
\end{itemize}
\end{hypothesis}
\begin{remark}\label{RM22}
The first hypothesis (DK) will bar the possibility of killing inside for desirable Markov process corresponding to $(I,\bs,\fm)$; see Remark~\ref{RM38}.  The second hypothesis (DM)  will be only used to obtain a fully supported ``regularized" speed measure in Theorem~\ref{THM6}.  
Note that $\text{supp}[\bs]\subset \text{supp}[\fm]$ implies that $(\alpha, \beta)\subset U$ for any $\fm$-negligible open interval $(\alpha,\beta)\subset I$.  When $\bs$ is strictly increasing,  (DM) reads as 
\[
\text{supp}[\fm]=[l,r],\text{ and }\fm(\{x\})>0,\; x\in D^0,
\]  as is identified with the assumption in \cite{S79}. 
\end{remark}

\subsection{Absolute continuity with respect to $\bs$}

%We aim to introduce a quadratic form associated to $(I,\bs,\fm)$,  which is expected to be the extension of Dirichlet form associated to regular diffusion on $I$.  The crucial concept is

In this subsection we introduce the concept of \emph{absolute continuity} with respect to $\bs$.  It is rather standard but for readers' convenience we state some details as follows.  For convenience we make the conventions $\int_{(0,x)}:=-\int_{[x,0)}$ and $\int_{(0,x]}:=-\int_{(x,0)}$ for $x<0$.  In addition,  for a measure $\mu$ on $I$ with $\mu(\{0\})=0$ and $x<0$,  $\mu((0,x)):=-\mu([x,0))$ and $\mu((0,x]):=-\mu((x,0))$. 

%Before moving on we prepare some elements concerning the absolute continuity with respect to $\bs$.  %A function $f$ that will be under consideration is usually defined on $(l,r)$.  When the limit $\lim_{x\rightarrow l\text{ or }r}f(x)$ exists,  we set $f(l)$ or $f(r)$ to be this limit. 
We first decompose $\bs$ into three parts,  each of which induces a Radon measure.  Put
\[
	\mu^+_d:=\sum_{x\in D^+}\left(\bs(x+)-\bs(x)\right)\cdot \delta_x,\quad \mu^-_d:=\sum_{x\in D^-}\left(\bs(x)-\bs(x-)\right)\cdot \delta_x,
\]
where $\delta_x$ is the Dirac measure at $x$,  and for $x\in I$,  
\begin{equation}\label{eq:22}
	\bs^+_d(x):=\mu^+_d((0,x)),\quad \bs^-_d(x):=\mu^-_d((0,x]),\quad \bs_c(x):=\bs(x)-\bs^+_d(x)-\bs^-_d(x).  
\end{equation}
Note that $\mu^\pm_d$ are positive Radon measures on $(l,r)$, $\bs^+_d$ (resp.  $\bs^-_d$) is left (resp.  right) continuous,  and $\bs_c$ is increasing and continuous on $I$.  Denote by $\mu_c$ the Lebesgue-Stieltjes measure of $\bs_c$.   
With the decomposition \eqref{eq:22} at hand we turn to introduce a family $\mathscr{S}$ of absolutely continuous functions.  The following definition is due to \cite{S79}.

\begin{definition}
A function $f$ on $I$ is called \emph{$\bs$-continuous} if $f$ has finite limits from the left and the right on $I$ and that the right or left continuity of $\bs$ at a point implies the same property of $f$.  Denote by $C_\bs(I)$ the family of all $\bs$-continuous functions on $I$.  
\end{definition}

For $f\in C_\bs(I)$,  set $f'_{\bs,\pm}(x):=(f(x)-f(x\pm))/(\bs(x)-\bs(x\pm))$ for $x\in D^\pm$.  Then 
\[
f^c(x):=f(x)-\int_{(0,x)} f'_{\bs,+}(y)\mu^+_d(dy)-\int_{(0,x]} f'_{\bs,-}(y)\mu^-_d(dy),\quad x\in I
\]
gives a continuous function on $I$.  Define further a family of absolutely continuous functions with respect to $\bs_c$:
 \[
 	\sS_c:=\left\{g\in C(I): g\ll \bs_c,  \frac{dg}{d\bs_c}\in L^2(I, \mu_c)\right\}
 \]
 and another two families of functions
 \[
 	\sS^+_d:=\left\{f=\int_{(0,\cdot)} gd\mu^+_d: g\in L^2(I,\mu^+_d)\right\},\quad \sS^-_d:=\left\{f=\int_{(0,\cdot]} gd\mu^-_d: g\in L^2(I,\mu^-_d)\right\}.
 \]
 Now we have a position to put forward the family $\sS$.  
 
 \begin{definition}
 Define a family of $\bs$-continuous functions on $I$:
 \begin{equation}\label{eq:11-2}
\sS:=\left\{f\in C_\bs(I): f'_{\bs,\pm}\in L^2(D^\pm, \mu^\pm_d),  f^c\in \sS_c\right\}.  
\end{equation}
For every $f,g\in \sS$,  define 
\begin{equation}\label{eq:11}
\int_I \frac{df}{d\bs}\frac{dg}{d\bs}d\bs:=\int_I \frac{df^c}{d\bs^c}\frac{dg^c}{d\bs^c}d\bs^c+\int_{D^+}f'_{\bs,+}g'_{\bs,+}d\mu^+_d+\int_{D^-}f'_{\bs,-}g'_{\bs,-}d\mu^-_d, 
\end{equation}
 where $g^c, g'_{\bs,\pm}$ are defined analogously.  
 \end{definition}
 \begin{remark}\label{RM26}
 Note that $f\in \sS$ is constant on each interval $(c_n,d_n)$ appearing in \eqref{eq:14}.   In addition,  every $f\in \sS$ admits a unique decomposition
 \[
 	f=f^c+f^++f^-,
 \]
 where $f^c\in \sS_c$,  $f^+:=\int_{(0,\cdot )} f'_{\bs,+}(y)\mu^+_d(dy)\in \sS^+_d$ and $f^-:=\int_{(0,\cdot]} f'_{\bs,-}(y)\mu^-_d(dy)\in \sS^-_d$.  Particularly,  when $r$ (resp. $l$) is approachable,  $f\in \sS$ admits a finite limit $f(r):=\lim_{x\uparrow r}f(x)$ (resp.  $f(l):=\lim_{x\downarrow l}f(x)$) even if $r\notin I$ (resp. $l\notin I$). 
 \end{remark}

Denote by 
\begin{equation}\label{eq:34}
	\overline{\bs([l,r])}=\bs([l,r])\cup \{\bs(x-): x\in D^-\}\cup \{\bs(x+): x\in D^+\}
\end{equation}	
 the closure of $\bs([l,r]):=\{\bs(x):x\in [l,r]\}$ in $\overline{\bR}$. 
Set $\widehat{l}:=\bs(l)$ and $\widehat{r}:=\bs(r)$.  We call $\widehat{l}$ or $\widehat{r}$ \emph{reflecting} if so is $l$ or $r$.  Let $\widehat{I}$ be a subset of $\overline{\bs([l,r])}$ such that
\begin{equation}\label{eq:27-3}
	\overline{\bs([l,r])}\setminus \{\widehat{l},\widehat{r}\}\subset \widehat{I},\quad \widehat{l}\in \widehat{I}\text{ (resp. }\widehat{r}\in \widehat{I}\text{) if and only if }\widehat{l}\text{ (resp. }\widehat{r}\text{) is reflecting.}
\end{equation}
Then $\widehat{I}$ is a \emph{nearly closed subset} of $\bR$,  i.e.  $\widehat{I}\in \mathscr K$.  
To make $\sS$ more comprehensible,  we set an auxiliary map
\begin{equation}\label{eq:27-2}
	\widehat{\iota}: \sS\rightarrow C_\mathrm{f}(\widehat{I}),\quad  f\mapsto \widehat{f},
\end{equation}
where $C_\mathrm{f}(\widehat{I})$ is the family of all continuous functions on $\widehat{I}$ such that $\widehat{f}(\widehat{j}):=\lim_{\widehat{x}\rightarrow \widehat{j}}\widehat{f}(\widehat{x})$ exists in $\bR$ whenever $\widehat{j}\notin \widehat{I}$ is finite for $\widehat{j}=\widehat{l}$ or $\widehat{r}$,  and $\widehat{f}$ is defined as
\begin{equation}\label{eq:35}
	\widehat{f}(\bs(x)):=f(x),\; x\in I, \quad
	\widehat{f}(\bs(x\pm)):=f(x\pm),\; x\in D^\pm.  
\end{equation}
Note that 
\begin{equation}\label{eq:36-2}
	\widehat{f}(\widehat{l})=f(l)\quad \left(\text{resp. }\widehat{f}(\widehat{r})=f(r) \right)
\end{equation}
if $\widehat{l}$ (resp.  $\widehat{r}$) is finite,  and the map \eqref{eq:27-2} is a well-defined injection due to Remark~\ref{RM26}.  
 Define
\begin{equation}\label{eq:211}
	\widehat{\sS}:=\{\widehat{f}=\widehat{\iota}f:  f\in \sS\}.    
\end{equation}
Then $\widehat{\iota}$ is bijective between $\sS$ and $\widehat{\sS}$. 
The expression of $\widehat{\mathscr S}$ below is crucial to proving the main result of this paper.  Its proof is elementary but long,  and we put it in Appendix~\ref{APPA}.   %Clearly if $r$ ($r^*$) is approachable,  then $f\in \sS$ ($f^*\in \sS^*$) is finite and left continuous at $r$ ($r^*$).  The right continuity at $r$ relies on $\bs(r-)=\bs(r)$.  

\begin{lemma}\label{LM39}
%Let $f\in \sS$ and $\widehat{f}=\widehat{\iota}f\in \widehat{\sS}$.  Then $f(r)=\widehat f(\widehat{r})$ (resp.  $f(l)=\widehat{f}(\widehat{l})$) if $r$ (resp.  $l$) is approachable.  Furthermore,  
It holds that
\begin{equation}\label{eq:33}
\widehat{\sS}=\left\{\widehat{f}=\widehat{h}|_{\widehat{I}}: \widehat{h}\in \dot H^1_e\left((\widehat{l},\widehat{r}) \right)\right\},
\end{equation}
where $\dot H^1_e((\widehat{l},\widehat{r}))$ is the family of all absolutely continuous function $\widehat{h}$ on $(\widehat{l},\widehat{r})$ such that $\widehat h'\in L^2((\widehat{l},\widehat{r}))$. 
\end{lemma}

%It is easy to verify that $f\in \sS$ has left/right limits in $(l,r)$,  and $f$ is continuous (resp.  left/right continuous) at $x$ if so is $\bs$.  
%When 

%\begin{remark}
%When $\bs$ is right continuous,  $\mu^+_d\equiv 0$ and $\mu_c+\mu^-_d$ is the Lebesgue-Stieltjes measure,  also denoted by $d\bs$, of $\bs$.  Meanwhile
%\[
%	\sS=\left\{f(\cdot )=c_0+\int_{(0,\cdot]}g(y)d\bs(y): g\in L^2(I,d\bs), c_0\in \bR\right\}.
%\]
%For $f\in \sS$,  $df/d\bs=df^c/d\mu_c$ on the continuous part of $\bs$ and $df/d\bs=df^-/d\mu^-_d=(f(x)-f(x-))/(\bs(x)-\bs(x-))$ whenever $\bs(x)\neq \bs(x-)$. 
%\end{remark}

\subsection{Dirichlet form associated to $(I,\bs, \fm)$}\label{SEC14}

%Let $I:=\langle l, r\rangle$ be the interval ended by $l$ and $r$ obtained by taking out non-reflecting endpoints from $I$.  Clearly $\fm$ is Radon on $I$ and we denote the restriction of $\fm$ to $I$ still by $\fm$.  
%Let $L^2(I,\fm)$ be the family of all functions $f$ defined on $I$ such that $f$ is $\fm$-measurable and square integrable on $(l,r)$ and $f(l)=0$ (resp.  $f(r)=0$) if $\fm(\{l\})=\infty$ (resp.  $\fm(\{r\})=\infty$).  We say a sequence of functions $f_n$ converges to $f$ in the sense of $\fm$-a.e. or $L^2(I,\fm)$ if 
What we are concerned with is the quadratic form $(\sE,\sF)$ on $L^2(I,\fm)$:
\begin{equation}\label{eq:25}
\begin{aligned}
	&\sF:= \{f\in \sS\cap L^2(I,\fm): f(j)=0\text{ if }j\notin I\text{ is approachable for }j=l\text{ or }r\},  \\
	&\sE(f,g):=\frac{1}{2}\int_I \frac{df}{d\bs}\frac{dg}{d\bs}d\bs,\quad f,g\in \sF,
\end{aligned}
\end{equation}
where $\sS$ is defined as \eqref{eq:11-2} and $\sE(f,g)$ is defined as \eqref{eq:11}.   It coincides with \eqref{eq:13-2} if $\bs$ is strictly increasing and continuous.  However in general,  \eqref{eq:25} is not even a Dirichlet form, because $f\in \sF$ is constant on $(c_n,d_n)$ and hence $\sF$ is not necessarily dense in $L^2(I,\fm)$.
Fortunately as proved in the following result,  $(\sE,\sF)$ is a Dirichlet form on $L^2(I,\fm)$ \emph{in the wide sense},  i.e.  $(\sE,\sF)$ is a non-negative symmetric closed form satisfying the Markovian property (cf. \cite[\S1.4]{FOT11}; while we do not assume that $\fm$ is fully supported).   In abuse of terminology we call \eqref{eq:25} the \emph{Dirichlet form associated to} $(I,\bs,\fm)$.  

\begin{theorem}\label{LM12}
The quadratic form $(\sE,\sF)$ defined as \eqref{eq:25} is a Dirichlet form on $L^2(I,\fm)$ in the wide sense.  
\end{theorem}
\begin{proof}
Firstly we note that $(\sE,\sF)$ is clearly a non-negative symmetric quadratic form on $L^2(I,\fm)$.  To prove its closeness,  take an $\sE_1$-Cauchy sequence $$f_n=f^c_n+f^+_n+f^-_n\in \sF,\quad n\geq 1.$$
We do not lose a great deal by assuming that $f_n$ converges to $f$ both $\fm$-a.e.  and in $L^2(I,\fm)$,
\begin{equation}\label{eq:12}
	g^c_n:=\frac{df^c_n}{d\mu_c}\rightarrow g^c\quad \text{ in }L^2(\mu_c)\text{ and }\mu_c\text{-a.e.}, 
\end{equation}
and
\begin{equation}\label{eq:27}
	g^\pm_n:=\frac{df^\pm_n}{d\mu^\pm_d}\rightarrow g^\pm\quad \text{ in }L^2(\mu^\pm_d)\text{ and }\mu^\pm_d\text{-a.e.}
\end{equation}
Assume further that $f(0)=\lim_{n\rightarrow\infty} f_n(0)=\lim_{n\rightarrow \infty} f^c_n(0)$. 
Set for $x\in I$,  
\[
\begin{aligned}
	&f^c(x):=f(0)+\int_{(0,x)}g^c(y)\mu_c(dy)\in \sS_c, \\
	&f^+(x):=\int_{(0,x)}g^+(y)\mu^+_d(dy)\in \sS^+_d, \quad f^-(x):=\int_{(0,x]}g^-(y)\mu^-_d(dy)\in \sS^-_d.
\end{aligned}\] 
%Since $\lim_{n\rightarrow \infty} f^c_n(0)=f(0)$,  it follows from  $\lim_{n\rightarrow \infty} f^{c,\pm}_n(x)=f^{c,\pm}(x)$ for $x\in I$.  
Define $\hat{f}:=f^c+f^++f^-\in \sS$.   Then \eqref{eq:12} and \eqref{eq:27} yield $\sE(f_n-\hat{f},f_n-\hat{f})\rightarrow 0$ and 
\begin{equation}\label{eq:28}
	f_n(x)\rightarrow \hat{f}(x),\quad x\in I.
\end{equation}
 % and $x=l$ or $r$ when $l$ or $r$ is approachable.  Particularly,  if ,  then $\hat{f}(j)=\lim_{n\rightarrow \infty}f_n(j)=0$.  
When $j\notin I$ is approachable for $j=l$ or $r$,  \eqref{eq:28} still holds for $x=j$.  Particularly $\hat{f}(j)=\lim_{n\rightarrow \infty}f_n(j)=0$,  and one eventually obtains $\hat{f}\in \sF$.  
To conclude the closeness,  we only need to note that $f=\hat{f}$,  $\fm$-a.e. on $I$,  by means of $f_n\rightarrow f$,  $\fm$-a.e.,  and \eqref{eq:28}. 
% Note that $f_n\rightarrow f$, $\fm$-a.e.  on $(l,r)$.  It follows from \eqref{eq:28} that $f=\hat{f}$,  $\fm$-a.e.  on $(l,r)$.  Now consider $j\in I$ with $\fm(\{j\})>0$ for $j=l$ or $r$.  When $j$ is not approachable,   we already define $\hat{f}(r)$ as $f(j)$.  When $j$ is approachable,  \eqref{eq:28} also yields that $\hat{f}(j)=f(j)$.  Hence $f=\hat f$,  $\fm$-a.e.,  is concluded.  

Finally it suffices to show the Markovian property of $(\sE,\sF)$.  %We adopt the notations in \S\ref{SEC32} and Lemma~\ref{LM39} will be employed in this step.  
Take $f\in \sF$ and let $g$ be a normal contraction of $f$,  i.e.  
\begin{equation}\label{eq:29}
	|g(x)-g(y)|\leq |f(x)-f(y)|,\; x,y\in I,\quad |g(x)|\leq |f(x)|,\;x\in I.  
\end{equation}
It is easy to verify that $g$ is $\bs$-continuous and constant on each $(c_n,d_n)$,  so that we can define a function $\widehat{g}$ on $\widehat{I}$ as \eqref{eq:35} with $g$ in place of $f$.   Further let $\widehat{f}$ be defined as \eqref{eq:35} for this $f$.  It follows from \eqref{eq:29} that 
\begin{equation}\label{eq:210}
|\widehat g(\widehat x)-\widehat g(\widehat y)|\leq |\widehat f(\widehat x)-\widehat f(\widehat y)|,\; \widehat x,\widehat y\in\widehat  I,\quad |\widehat g(\widehat x)|\leq |\widehat f(\widehat x)|,\;\widehat x\in \widehat I.  
\end{equation}
Denote by $\widehat{J}:=\langle \widehat{l}, \widehat{r}\rangle$ an interval ended by $\widehat{l}$ and $\widehat{r}$ such that $\widehat{l}\in \widehat{J}$ (resp. $\widehat{r}\in \widehat{J}$) if and only if $\widehat{l}\in \widehat{I}$ (resp. $\widehat{r}\in \widehat{I}$).  Then \eqref{eq:36-2} and Lemma~\ref{LM39} yield that $\widehat{f}$ belongs to
\begin{equation}\label{eq:217}
	\left\{\widehat{h}|_{\widehat{I}}: \widehat{h}\in \dot H^1_e\left((\widehat{l},\widehat{r}) \right), \widehat{h}(\widehat{j})=0\text{ whenever }\widehat{j}\in \bR\setminus \widehat{J}\text{ for }\widehat{j}=\widehat{l}\text{ or }\widehat{r}\right\}.
\end{equation}
Note that $\widehat{I}$ is a closed subset of $\widehat{J}$ and,  in view of \cite[Theorem~6.2.1~(2)]{FOT11},  the family \eqref{eq:217} is the extended Dirichlet space of the time-changed Brownian motion on $\widehat{I}$ (see also Theorem~\ref{THM19}).  Particularly,  every normal contraction operates on \eqref{eq:217}, and it follows from \eqref{eq:210} that $\widehat{g}$ also belongs to \eqref{eq:217}.  Thus $\widehat{g}\in  \widehat{\sS}$ and 
%Using \eqref{eq:210},   ,   and applying ,  we can obtain that $\widehat{g}\in \widehat{\sS}$.  Note that $\widehat\iota$ is a bijective between $\sS$ and $\widehat{\sS}$.  Thus 
$g=\widehat{\iota}^{-1}\widehat{g}\in \sS$.  By means of \eqref{eq:29},  one can also get $g\in L^2(I,\fm)$ and if $j\notin I$ is approachable for $j=l$ or $r$,  then
\[
	|g(j)|=\lim_{x\rightarrow j}|g(x)|\leq \lim_{x\rightarrow j}|f(x)|=|f(j)|=0.  
\]
Consequently $g\in \sF$ is concluded.  In addition,  \eqref{eq:210} also yields that $|\widehat{g}'|\leq |\widehat{f}'|$,  a.e.  on $\widehat{I}$.  Hence $\sE(g,g)\leq \sE(f,f)$ follows from \eqref{eq:210},  \eqref{eq:38} and \eqref{eq:39-2}.  That completes the proof. 
\end{proof}
%\begin{remark}\label{RM13}
%When $r$ is approachable,  it also holds that $(\sE,\sF^0)$,  where $\sF^0=\{f\in \sF: f(r)=0\}$, is a Dirichlet form on $L^2(I,m)$ in the wide sense.   In fact,  adopt the notations in this proof and suppose further $f_n(r)=0$.  Then we have
%\[
%	f^c(r)=f(0)+\int_{(0,r)}g^c(y)\mu_c(dy)=\lim_{n\rightarrow \infty} f^c_n(r).  
%\]
%Analogously $f^\pm(r)=\lim f^\pm_n(r)$.  Hence $f(r)=\lim f_n(r)=0$.  
%\end{remark}

%In the next section we will regularize $(I,\bs,\fm)$ and $(\sE,\sF)$ by a certain way,  so that a Hunt process is obtained.  This process,  called (image) regularized Markov process associated to $(I,\bs,\fm)$, will be further explored in the subsequent sections.  
%Then in \S\ref{SEC9} we will come back to $(\sE,\sF)$ for the special case that $\bs$ is strictly increasing.  In this case $(\sE,\sF)$ is a (not regular) Dirichlet form on $L^2(I,\fm)$ and there is an $\fm$-symmetric continuous simple (not strong) Markov process on $I$ whose Dirichlet form is $(\sE,\sF)$. 

\section{Regular representations and regularized Markov processes}\label{SEC3}

Following \cite{F71},  we call $(E_1,\fm_1,\sE^1,\sF^1)$ a \emph{D-space} provided that $(\sE^1,\sF^1)$ is a Dirichlet form on $L^2(E_1,\fm_1)$ in the wide sense.  The space $\sF^1_b:=\sF^1\cap L^\infty(E_1,\fm_1)$ is an algebra over $\bR$.  Let $(E_2,\fm_2,\sE^2,\sF^2)$ be another D-space.  Then $(E_1,\fm_1,\sE^1,\sF^1)$ and $(E_2,\fm_2,\sE^2,\sF^2)$ are called \emph{equivalent} if there is an algebra isomorphism $\Phi$ from $\sF^1_b$ to $\sF^2_b$ such that $\Phi$ preserves three kinds of metrics: For $f\in \sF^1_b$, 
\begin{equation}\label{eq:31}
	\|f\|_\infty=\|\Phi(f)\|_\infty, \quad (f,f)_{\fm_1}=(\Phi(f),\Phi(f))_{\fm_2},\quad \sE^1(f,f)=\sE^2(\Phi(f),\Phi(f)),
\end{equation}
where $\|\cdot\|_\infty:=\|\cdot\|_{L^\infty(E_i,\fm_i)}$ and $(\cdot,\cdot)_{\fm_i}=(\cdot,\cdot)_{L^2(E_i,\fm_i)}$ for $i=1,2$.  In addition,  $(E_2,\fm_2,\sE^2,\sF^2)$ is called a \emph{regular representation} of $(E_1,\fm_1,\sE^1,\sF^1)$ if they are equivalent and $(\sE^2,\sF^2)$ is regular on $L^2(E_2,\fm_2)$.  The main result of \cite{F71} shows that every D-space admits regular representations, and all regular representations are quasi-homeomorphic.  

As obtained in Theorem~\ref{LM12},  $(I,\fm,\sE,\sF)$ is a D-space.  This section is 
%Let $(I,\bs, \fm)$ be a triple satisfying (DK) and \text{(DM)} as stated in \S\ref{SEC2}.  This section is
 devoted to finding out a \emph{canonical} regular representation for it.  %The associated $\fm^*$-symmetric Markov process of $(\sE^*,\sF^*)$ will be called the \emph{regularized Markov process} associated to $(I,\bs,\fm)$.  
Furthermore,  we will prove that the quasi-homeomorphism between two regular representations of $(I,\fm,\sE,\sF)$ is essentially a strict homoemorphism. 

\subsection{Canonical regular representation}\label{SEC32}

Let us prepare some ingredients for introducing the canonical regular representation.  Recall that $\overline{\bs([l,r])}$ is defined as \eqref{eq:34} and $\widehat{I}$ is defined as \eqref{eq:27-3}.  Note that $(\widehat{l},\widehat{r})\setminus \widehat{I}$ is open and thus can be written as a disjoint union of open intervals:
\begin{equation}\label{eq:15}
	(\widehat{l},\widehat{r})\setminus \widehat{I}=\bigcup_{k\geq 1}(\widehat{a}_k,\widehat{b}_k).  
\end{equation}
  Let $\widehat{\fm}:= \fm\circ \bs^{-1}$ be the image measure of $\fm$ under the embedding map $\bs: [l,r]\rightarrow \overline{\bs([l,r])}$.  Put $\widehat{\bs}(\widehat{x})=\widehat{x}$ for $\widehat{x}\in \widehat{I}$.  The triple $(\widehat{I},\widehat{\bs},\widehat{\fm})$ is called the \emph{canonical regularization} of $(I,\bs,\fm)$.  

%\begin{definition}\label{DEF31}
%The  is defined as the triple $(\widehat{I}, \widehat{\bs},\widehat{\fm})$:
%\begin{itemize}
%\item[(i)] $\overline{\bs([l,r])}\setminus \{\widehat{l},\widehat{r}\}\subset \widehat{I}\subset \overline{\bs([l,r])}$, and $\widehat{l}\in \widehat{I}$ (resp. $\widehat{r}\in \widehat{I})$) if and only if $\widehat{l}$ (resp.  $\widehat{r}$) is reflecting;
%\item[(ii)] ;
%\item[(iii)] $\widehat{\fm}$ is the restriction of \eqref{eq:33-2} to $\widehat{I}$.
%\end{itemize}
%  where $\widehat{I}$ is the subset of $\overline{\bs([l,r])}$ obtained by taking out non-reflecting endpoints from $\overline{\bs([l,r])}$,   and .  
%\end{definition}
%\begin{remark}\label{RM34}
%We point out that if $\widehat{l}\in \widehat{I}$,  then $\widehat{l}\in \bR$.  In addition,  if $\widehat{l}\notin \widehat{I}$,  then $\widehat{l}$ can be approximated by points in $\widehat{I}$.  Analogical facts hold for $\widehat{r}$.  Particularly,   $\widehat{I}$ is a nearly closed subset of $\bR$,  i.e.  $\widehat{I}\in \mathscr K$.  In addition 
%Both $\widehat{a}_k$ and $\widehat{b}_k$ are finite. 
%\end{remark}

With $(\widehat{I}, \widehat{\bs},\widehat{\fm})$ at hand,  we raise another quadratic form as follows:  
\begin{equation}\label{eq:38-3}
\begin{aligned}
	&\widehat\sF:=\{\widehat f=\widehat\iota f: f\in \sF\},\\
	&\widehat\sE(\widehat f,\widehat g):=\sE(\widehat{\iota}^{-1} \widehat f,\widehat{\iota}^{-1} \widehat g),\quad \widehat f, \widehat g \in \widehat \sF.
\end{aligned}
\end{equation}
Recall that $\sS$ and $\widehat{\sS}$ are defined as \eqref{eq:11-2} and \eqref{eq:211} respectively,  and $\widehat{\iota}$ is a bijection between them.     
The main result of this section is as follows. 

\begin{theorem}\label{THM6}
Let $(\widehat{I},\widehat{\bs},\widehat{\fm})$ be the canonical regularization of $(I,\bs,\fm)$ and $(\widehat{\sE},\widehat{\sF})$ be defined as \eqref{eq:38-3}.  Then $\widehat{\fm}$ is a fully supported Radon measure on  $\widehat{I}$,  and $(\widehat{\sE},\widehat{\sF})$ is a regular and irreducible Dirichlet form on $L^2(\widehat I,\widehat \fm)$ admitting the representation:
\begin{equation}\label{eq:321}
\begin{aligned}
	&\widehat{\sF}=\left\{\widehat f\in L^2(\widehat I,\widehat\fm)\cap \widehat \sS: \widehat f(\widehat j)=0\text{ if }\widehat j\in \bR\setminus \widehat{I}\text{ for }\widehat j=\widehat l\text{ or }\widehat r\right\},  \\
	&\widehat{\sE}(\widehat{f},\widehat{f})=\frac{1}{2}\int_{\widehat{I}} \widehat f'(\widehat x)^2d\widehat x+\frac{1}{2}\sum_{k\geq 1}\frac{\left(\widehat{f}(\widehat{a}_k)-\widehat{f}(\widehat{b}_k)\right)^2}{|\widehat{b}_k-\widehat{a}_k|},\quad \widehat{f}\in \widehat{\sF},
\end{aligned}
\end{equation}
where $\widehat{a}_k,\widehat{b}_k$ appear in \eqref{eq:15}.  Particularly,  $(\widehat{I},\widehat{\fm}, \widehat{\sE},\widehat{\sF})$ is a regular representation of $(I,\fm,\sE,\sF)$. 
\end{theorem}
\begin{proof}
To prove that $\widehat{\fm}$ is fully supported,  argue by contradiction and suppose that  for some $\widehat a<\widehat b$ with $(\widehat a, \widehat b)\cap \widehat{I}\neq \emptyset$, 
\[
	\widehat{\fm}\left((\widehat a,\widehat b)\cap \widehat{I}\right)=0.  
\]
Note that an isolated point in $\widehat{I}$ must be $\bs(x)$ for some $x\in D^0$ or $\bs((c_n,d_n))$ for some isolated interval $(c_n,d_n)$ in \eqref{eq:14}.  Hence \text{(DM)} yields that  $(\widehat a, \widehat b)\cap \widehat{I}$ contains no isolated points.  Take $\widehat{x}\in (\widehat{a},\widehat{b})\cap \widehat{I}$.  There is a sequence $\widehat{x}_p=\bs(x_p)$ with $x_p\in I$ such that $\widehat{x}_p\rightarrow \widehat{x}$.  We may and do assume that $\widehat{x}_p, \widehat{x}_{p+1},  \widehat x_{p+2}, \widehat{x}_{p+3}\in (\widehat{a},\widehat{b})$ and $\widehat{x}_p< \widehat{x}_{p+1}<\widehat{x}_{p+2}<\widehat{x}_{p+3}$ for some $p$.  It follows that $x_p<x_{p+1}<x_{p+2}<x_{p+3}$ and $\{\bs(y): y\in (x_p,x_{p+3})\}\subset (\widehat{a},\widehat{b})\cap \widehat{I}$. 
Hence $\fm((x_p,x_{p+3}))\leq \widehat{\fm}((\widehat{a},\widehat{b})\cap \widehat{I})=0$.  In view of Remark~\ref{RM22},  $(x_p,x_{p+3})\subset (c_n,d_n)$ for some $(c_n,d_n)$ in \eqref{eq:14}.  Particularly,  $$\widehat{x}_{p+1}=\bs(x_{p+1})=\bs(x_{p+2})=\widehat{x}_{p+2},$$ as leads to the contradiction with $\widehat{x}_{p+1}<\widehat{x}_{p+2}$.  
Now we turn to prove that $\widehat{\fm}$ is Radon on $\widehat{I}$.  Let $\widehat{J}:=\langle \widehat{l},  \widehat{r}\rangle$ be the interval ended by $\widehat{l}$ and $\widehat{r}$ such that $\widehat{l}\in \widehat{J}$ ($\widehat{r}\in \widehat{J}$) if and only if $\widehat{l}\in \widehat{I}$ ($\widehat{r}\in \widehat{I}$).  It suffices to show that $\widehat{\fm}([\widehat a,\widehat b])<\infty$ for any interval $[\widehat a,\widehat b]\subset \widehat{J}$.  To accomplish this,  set 
\[
	a:=\inf\{x\in I: \bs(x)\geq  \widehat{a}\},\quad b:=\sup\{x\in I: \bs(x)\leq \widehat{b}\}.  
\]
Then $\{x\in I: \bs(x)\in [\widehat{a},\widehat{b}]\}\subset [a,b]$.  Note that if $a=l$,  then $\widehat{l}\leq \widehat{a}\leq \bs(l+)=\bs(l)=\widehat{l}$.  Thus $\widehat{l}=\widehat{a}\in \widehat{I}$ must be a reflecting endpoint of $\widehat{I}$.  Accordingly $l \in I$ is reflecting.  Analogously if $b=r$,  then $r\in I$ is reflecting.  As a result,  $\widehat{\fm}([\widehat{a},\widehat{b}])=\fm\circ \bs^{-1}([\widehat{a},\widehat{b}])\leq \fm([a,b])<\infty$.  

The rest is to establish the canonical regular representation of $(I,\fm, \sE,\sF)$.  
 Firstly we formulate \eqref{eq:321}.  The expression of $\widehat{\sE}(\widehat{f},\widehat{f})$ is a consequence of  \eqref{eq:38} and \eqref{eq:39-2}.  It suffices to prove the first identity in \eqref{eq:321}.  Denote the family on its right hand side by $\widehat\sG$.  

Take $f\in \sF$ and let $\widehat{f}:=\widehat{\iota}f\in \widehat{\sF}$.  Then $\widehat{f}\in \widehat{\sS}\subset C_\mathrm{f}(\widehat{I})$.  Suppose $\widehat{r}\notin \widehat{I}$ and we will show $\widehat{f}|_{[0, \widehat{r})}\in L^2([0, \widehat{r}),\widehat{\fm})$,  so that $\widehat{f}\in L^2(\widehat{I},\widehat{\fm})$ can be obtained by virtue of $\widehat f\in C_\mathrm{f}(\widehat{I})$ and that $\widehat{\fm}$ is Radon on $\widehat{I}$.  To do this,  assume without lose of generality that $0\notin \{c_n,d_n:n\geq 1\}$.  Note that
\[
\int_{[0,\widehat{r})}\widehat{f}(\widehat{x})^2\widehat{\fm}(d\widehat{x})=\int_{\{x\in I:  \bs(x)\in [0,\widehat{r})\}} f(x)^2\fm(dx).
\]	
When $r\notin U$,  i.e.  $\bs(x)$ is strictly increasing as $x\uparrow r$,  we have
\begin{equation}\label{eq:39-3}
	\{x\in I:  \bs(x)\in [0,\widehat{r})\}=[0,  r).  
\end{equation}
Otherwise if $r=d_n$ for some $n$,  then $\bs(x)=\widehat{r}<\infty$ for $x\in (c_n,d_n]$ and 
\begin{equation}\label{eq:318}
[0, c_n)\subset 	\{x\in I:  \bs(x)\in [0,\widehat{r})\}\subset [0,c_n].  
\end{equation}
For either case we can obtain $\int_{[0,\widehat{r})}\widehat{f}(\widehat{x})^2\widehat{\fm}(d\widehat{x})<\infty$ by means of $f\in L^2(I,\fm)$.  On the other hand,  if $\widehat{r}\in\bR\setminus \widehat{I}$,  then $r$ is approachable but not reflecting.  Thus $\fm(r-)=\infty$ or $\fm(\{r\})=\infty$.  We always have $f(r)=0$ by the definition of $\sF$,  and because of \eqref{eq:36-2},  $\widehat{f}(\widehat{r})=f(r)=0$.  Therefore $\widehat{\sF}\subset \widehat{\sG}$ is concluded.  

To the contrary,  let $\widehat{f}\in \widehat{\sG}\subset \widehat{\sS}$.   Then there exists $f\in \sS$ such that $\widehat{f}=\widehat{\iota}f$.  We need to prove that $f\in \sF$.  To show $f\in L^2(I,\fm)$,  we assert that $f|_{[0,r]\cap I}\in L^2([0,r]\cap I,\fm)$. ($f|_{[l,0]\cap I}\in L^2([l,0]\cap I,\fm)$ can be obtained similarly.)  This trivially holds when $r$ is reflecting.  Suppose $r$ is not reflecting,  equivalently $\widehat{r}\notin \widehat{I}$.  In the case $r\notin U$,  \eqref{eq:39-3} tells us that 
\[
\int_{[0,r)}f(x)^2\fm(dx)=\int_{[0,\widehat{r})}\widehat{f}(\widehat{x})^2\widehat{\fm}(d\widehat{x})<\infty.  
\]
In the case \eqref{eq:318},  $\widehat{r}\in \bR\setminus \widehat{I}$ and hence $f(r)=\widehat{f}(\widehat{r})=0$ due to \eqref{eq:36-2} and $\widehat{f}\in \widehat{\sG}$.  Particularly $f(x)=0$ for any $x\in (c_n,d_n]$.  As a result one can verify that $\int_{[0,r)}f(x)^2\fm(dx)<\infty$.  Finally it suffices to prove $f(r)=0$ if $r\notin I$ is approachable.  Since $\widehat{r}$ is finite but not reflecting,  i.e.  $\widehat{r}\in \bR\setminus \widehat{I}$,  it follows from \eqref{eq:36-2} and $\widehat{f}\in \widehat{\sG}$ that $f(r)=\widehat{f}(\widehat{r})=0$.  Eventually we obtain $f\in \sF$, and $\widehat{\sG}\subset \widehat{\sF}$ is concluded. 

Repeating the argument in the proof of Theorem~\ref{LM12},  one can easily verify that $(\widehat{\sE},\widehat{\sF})$ is a Dirichlet form on $L^2(\widehat{I},\widehat{\fm})$ in the wide sense.  We turn to prove its regularity.  This will be completed by treating several cases separately as follows.  (Another proof for regularity involving transformation of time change will be shown in Theorem~\ref{THM19}.)  

\emph{Both $\widehat l$ and $\widehat r$ are finite}.   Clearly $\widehat{\sF}\subset C_\infty(\widehat{I})$,  where $C_\infty(\widehat{I})$ stands for the family of continuous functions vanishing at each endpoint not contained in $\widehat{I}$.  
 To prove the regularity of $(\widehat{\sE},\widehat{\sF})$ on $L^2(\widehat{I},\widehat{\fm})$,  in view of \cite[Lemma~1.3.12]{CF12},  we only need to prove that $\widehat{\sF}$ separates the points in $\widehat{I}$.  Fix $\widehat{x}, \widehat{y}\in \widehat{I}$ with $\widehat{x}\neq \widehat{y}$.  Let $\widehat{J}:=\langle \widehat{l},  \widehat{r}\rangle$ be the interval as in the first paragraph of this proof, and take $\widehat{f}\in C_c^\infty(\widehat{J})$ such that $\widehat f(\widehat t)=\widehat t$ for $\widehat t\in [\widehat{x}, \widehat{y}]$.  Then $\widehat{f}|_{\widehat{I}}\in \widehat{\sF}$ separates $\widehat{x}$ and $\widehat{y}$.  

\emph{Neither $\widehat l$ nor $\widehat r$ is finite}.  In this case $\widehat{l}=-\infty$,  $\widehat{r}=\infty$,  and $\widehat{J}=\bR$.  Mimicking the argument in the previous case,  we can conclude that $\widehat{\sF}\cap C_c(\widehat{I})$ is dense in $C_c(\widehat{I})$ with respect to the uniform norm.  To show the $\widehat{\sE}_1$-denseness of $\widehat{\sF}\cap C_c(\widehat{I})$ in $\widehat{\sF}$,  fix a bounded function $\widehat{f}\in \widehat{\sF}$ with $M:=\sup_{\widehat x\in \widehat{I}}|\widehat{f}(\widehat{x})|<\infty$.   Take a sequence of functions $\varphi_n\in C_c^\infty(\bR)$ such that
\begin{equation}\label{eq:311}
\begin{aligned}
&\varphi_n(\widehat{x})=1\quad \text{for }|\widehat{x}|< n;  \quad \varphi_n(\widehat{x})=0\quad \text{for } |\widehat{x}|>2n+1;\\
&|\varphi'_n(\widehat{x})|\leq 1/n,\quad n\leq |\widehat{x}|\leq 2n+1;\quad 0\leq \varphi_n(\widehat{x})\leq 1,\quad \widehat{x}\in \bR.  
\end{aligned}\end{equation}
Set $\widehat{f}_n:=\widehat{f}\cdot \varphi_n|_{\widehat{I}}$.  Since $\varphi_n|_{\widehat{I}}\in \widehat{\sF}$,  it follows that $\widehat{f}_n\in \widehat{\sF}\cap C_c(\widehat{I})$.  Clearly $\widehat{f}_n$ converges to $\widehat{f}$ in $L^2(\widehat{I},\widehat{\fm})$.  We prove that $\widehat{\sE}(\widehat{f}_n-\widehat{f},\widehat{f}_n-\widehat{f})\rightarrow 0$.  Denote by $B_R:=\{\widehat{x}: |\widehat{x}|<R\}$ for $R>0$.  In view of \eqref{eq:321},  $2\widehat{\sE}(\widehat{f}_n-\widehat{f},\widehat{f}_n-\widehat{f})$ is not greater than $A^1_n+A^2_n+A^3_n$, where
\[
	A^1_n:=\int_{\widehat{I}\cap B_{2n+1}^c} \widehat{f}'(\widehat{x})^2d\widehat{x}+\sum_{k: \widehat{b}_k>2n+1\text{ or }\widehat{a}_k<-2n-1}\frac{\left(\widehat{f}(\widehat{a}_k)-\widehat{f}(\widehat{b}_k)\right)^2}{|\widehat{b}_k-\widehat{a}_k|},
\]
\[
A^2_n:=\int_{\widehat{I} \cap (B_{2n+1}\setminus B_n)} \left(\widehat{f}'(\widehat{x})(\varphi_n(\widehat{x})-1)+\widehat{f}(\widehat{x})\varphi'_n(\widehat{x})\right)^2d\widehat{x}
\]
and 
\[
	A^3_n:=\sum_{k: n\leq |\widehat{a}_k|\leq 2n+1\text{ or }n\leq |\widehat{b}_k|\leq 2n+1}\frac{\left(\left(\widehat{f}\cdot(\varphi_n-1)\right)(\widehat{a}_k)-\left(\widehat{f}\cdot(\varphi_n-1)\right)(\widehat{b}_k)\right)^2}{|\widehat{b}_k-\widehat{a}_k|}.  
\]
Clearly $A^1_n\rightarrow 0$ as $n\rightarrow \infty$.  Regarding the second term,  we have by means of \eqref{eq:311} that
\[
	A^2_n\leq 2\int_{\widehat{I} \cap (B_{2n+1}\setminus B_n)}\left(\widehat{f}'(\widehat{x})^2+\frac{\widehat{f}(\widehat{x})^2}{n^2} \right)d\widehat{x}\rightarrow 0.  
\]
Note that $A^3_n\leq 2A^{31}_n+2A^{32}_n+2A^{33}_n$,  where
\[
	A^{31}_n:= \sum_{k: n\leq |\widehat{a}_k|\leq 2n+1\text{ or }n\leq |\widehat{b}_k|\leq 2n+1}\frac{\left(\widehat{f}(\widehat{a}_k)-\widehat{f}(\widehat{b}_k)\right)^2}{|\widehat{b}_k-\widehat{a}_k|}\rightarrow 0,
\]
\[
	A^{32}_n:=\sum_{k: n\leq |\widehat{a}_k|\leq 2n+1\text{ or }n\leq |\widehat{b}_k|\leq 2n+1}\frac{\left(\widehat{f}(\widehat{a}_k)\varphi_n(\widehat{a}_k)-\widehat{f}(\widehat{b}_k)\varphi_n(\widehat{a}_k)\right)^2}{|\widehat{b}_k-\widehat{a}_k|}\leq A^{31}_n\rightarrow 0,
\]
and 
\[
\begin{aligned}
A^{33}_n&:=\sum_{k: n\leq |\widehat{a}_k|\leq 2n+1\text{ or }n\leq |\widehat{b}_k|\leq 2n+1}\frac{\left(\widehat{f}(\widehat{b}_k)\varphi_n(\widehat{a}_k)-\widehat{f}(\widehat{b}_k)\varphi_n(\widehat{b}_k)\right)^2}{|\widehat{b}_k-\widehat{a}_k|} \\
&\leq \frac{M^2}{n^2}\sum_{k: n\leq |\widehat{a}_k|\leq 2n+1\text{ or }n\leq |\widehat{b}_k|\leq 2n+1}|\widehat{a}_k-\widehat{b}_k|\rightarrow 0.
\end{aligned}\]
Therefore $\widehat{\sE}(\widehat{f}_n-\widehat{f},\widehat{f}_n-\widehat{f})\rightarrow 0$ is concluded.  

The reminder cases can be treated analogously and we eventually conclude that $(\widehat{\sE},\widehat{\sF})$ is a regular Dirichlet form on $L^2(\widehat{I},\widehat{\fm})$.   

Next we prove the irreducibility of $(\widehat{\sE},\widehat{\sF})$.  Argue by contradiction and suppose that $\widehat A\subset \widehat{I}$ is a non-trivial $\{\widehat T_t\}$-invariant set,  i.e.  $\widehat\fm(\widehat A)\neq 0$ and $\widehat\fm(\widehat{I}\setminus \widehat A)\neq 0$,  where $\{\widehat T_t\}$ is the $L^2$-semigroup of $(\widehat{\sE},\widehat{\sF})$.  Then there exists a closed interval $0\in [\widehat l_0,\widehat r_0]\subset \widehat J$,  $\widehat{l}_0, \widehat{r}_0\in \widehat{I}$,  such that 
\begin{equation}\label{eq:320}
	\widehat\fm( [\widehat l_0,\widehat r_0] \cap \widehat A)>0,\quad \widehat \fm([\widehat l_0,\widehat r_0]\setminus \widehat A)>0.  
\end{equation}
Take another closed interval $[\widehat{l}_1,\widehat{r}_1]\subset \widehat{J}$ such that $[\widehat{l}_0,\widehat{r}_0]\subset [\widehat{l}_1,\widehat{r}_1]$ and $\widehat{l}_1<\widehat{l}_0$ (resp.  $\widehat{r}_1>\widehat{r}_0$) unless $\widehat{l}_0=\widehat{l}\in \widehat{I}$ (resp.  $\widehat{r}_0=\widehat{r}\in \widehat{I}$).  Note that $\{\widehat{h}|_{\widehat{I}}: \widehat{h}\in C_c^\infty(\widehat{J})\}\subset \widehat{\sF}$ due to Lemma~\ref{LM39}. Take a function $\widehat{f}=\widehat h|_{\widehat{I}}\in \widehat{\sF}$ with $\widehat{h}\in C_c^\infty(\widehat{J})$ and $\widehat{h}=1$ on $[\widehat l_1,\widehat r_1]$.  In view of \cite[Proposition~2.1.6]{CF12},  $\widehat{f}\cdot 1_{\widehat{A}}\in \widehat{\sF}$ and on account of Lemma~\ref{LM39},  $\widehat{f}\cdot 1_{\widehat{A}}\in \widehat{\sF}$ admits a continuous $\fm$-a.e. version denoted by $\tilde{f}_1$.  Clearly $\tilde{f}_1=0$ or $1$ pointwisely on $[\widehat l_1,\widehat r_1]\cap \widehat{I}$.  Consider the family of intervals:
\[
	\mathscr{I}:=\{(\widehat{a}_k,\widehat{b}_k)\subset [\widehat{l}_0,\widehat{r}_0]: k\geq 1\}.  
\]
When $\mathscr I$ is empty,  $\tilde{f}_1$ must be constant on $[\widehat{l}_0,\widehat{r}_0]$,  as leads to a contradiction with \eqref{eq:320}.  Now we consider $\mathscr I\neq \emptyset$ and assert that there exists $(\widehat{a}_k,\widehat{b}_k)\in \mathscr I$ such that $\tilde{f}_1(\widehat{a}_k)\neq \tilde{f}_1(\widehat{b}_k)$.  To accomplish this, set
\[
\begin{aligned}
\widehat\alpha&:=\sup\{\widehat{x}\in [\widehat{l}_0,0]\cap \widehat{I}: \tilde{f}_1(\widehat{x})\neq \tilde{f}_1(0) \}, \\
\widehat\beta&:=\inf\{\widehat{x}\in [0,\widehat{r}_0]\cap \widehat{I}: \tilde{f}_1(\widehat{x})\neq \tilde{f}_1(0)\},
\end{aligned}\]
where $\sup \emptyset:=-\infty$ and $\inf \emptyset:=\infty$.  Then \eqref{eq:320} implies that $\widehat\alpha \neq -\infty$ or $\widehat\beta \neq \infty$.  The former case leads to $\widehat\alpha=\widehat{a}_p$ for some $p$ and the latter one leads to $\widehat\beta=\widehat{b}_q$ for some $q$ by virtue of the continuity of $\tilde{f}_1$.  We have either $\tilde{f}_1(\widehat{a}_p)\neq \tilde{f}_1(\widehat{b}_p)$ or $\tilde{f}_1(\widehat{a}_q)\neq \tilde{f}_1(\widehat{b}_q)$.  Therefore the existence of such $(\widehat{a}_k,\widehat{b}_k)$ is obtained.  Without loss of generality assume that
\[
	\widehat{f}_1(\widehat{a}_k)=1,\quad \widehat{f}_1(\widehat{b}_k)=0.  
\]
Note that there exists $\widehat{\varepsilon}>0$ such that $\widehat{a}_k+\widehat{\varepsilon}<\widehat{b}_k-\widehat{\varepsilon}$ and 
\[
\begin{aligned}
	&(\widehat{a}_k-\widehat{\varepsilon},  \widehat{a}_k+\widehat{\varepsilon})\cap \widehat{I}\subset  \widehat{A},\quad \widehat{\fm}\text{-a.e.},  \\
	&(\widehat{b}_k-\widehat{\varepsilon},  \widehat{b}_k+\widehat{\varepsilon})\cap \widehat{I}\subset  \widehat{I}\setminus \widehat{A},\quad \widehat{\fm}\text{-a.e.},
\end{aligned}\]
as can be obtained by means of the continuity of $\tilde{f}_1$ and 
\[
	\{\widehat{x}\in [\widehat{l}_1,\widehat{r}_1]\cap \widehat{I}: \tilde{f}_1(\widehat{x})=1\}=\widehat{A}\cap [\widehat{l}_1,\widehat{r}_1],\quad \widehat{\fm}\text{-a.e. }
\]
When $\widehat{a}_k>\widehat{l}$ (resp.  $\widehat{b}_k<\widehat{r}$),  we may and do assume that $\widehat{a}_k-\widehat{\varepsilon}>\widehat{l}$ (resp.  $\widehat{b}_k+\widehat{\varepsilon}<\widehat{r}$).
Take another function $\widehat{g}\in \widehat{\sF}$ such that 
\[
\begin{aligned}
&\widehat{g}|_{(\widehat{a}_k-\widehat{\varepsilon}/2,  \widehat{a}_k+\widehat{\varepsilon}/2)\cap \widehat{I}}\equiv 1, \quad \widehat{g}|_{(\widehat{b}_k-\widehat{\varepsilon}/2,  \widehat{b}_k+\widehat{\varepsilon}/2)\cap \widehat{I}}\equiv 1, \\
&\widehat{g}|_{\widehat{I}\setminus \left((\widehat{a}_k-\widehat{\varepsilon},  \widehat{a}_k+\widehat{\varepsilon}) \cup (\widehat{b}_k-\widehat{\varepsilon},  \widehat{b}_k+\widehat{\varepsilon}) \right)}\equiv 0.  
\end{aligned}\]
Using \cite[Proposition~2.1.6]{CF12},  we get that $\widehat{g}_1:=\widehat{g}\cdot 1_{\widehat{A}}\in \widehat{\sF},  \widehat{g}_2:=g-\widehat{g}_1\in \widehat{\sF}$ and $\widehat{\sE}(\widehat{g}_1,\widehat{g}_2)=0$.  However,  in view of \eqref{eq:321},  a computation yields that
\[
	\widehat{\sE}(\widehat{g}_1,\widehat{g}_2)=-\frac{1}{2|\widehat{b}_k-\widehat{a}_k|}\neq 0
\]
leading to a contradiction.  

Finally it suffices to show that $(\widehat{I},\widehat{\fm},\widehat{\sE},\widehat{\sF})$ is a regular representation of $(I,\fm,\sE,\sF)$.  Denote by $\Phi$ the restriction of $\widehat{\iota}$ to $\sF_b$.  In view of $\widehat{\fm}=\fm\circ \bs^{-1}$, \eqref{eq:35} and \eqref{eq:38-3},  one may easily get that $\Phi$ is an algebra isomorphism and $\Phi \sF_b\subset \widehat{\sF}_b$.  To verify \eqref{eq:31},  we only consider the case $l\notin I, r\in I$.  The other cases can be treated analogically.  The last identity in \eqref{eq:31} is the consequence of \eqref{eq:38-3}.  If $l\neq c_n$,  then 
\[
	\widehat{I}=\bs(I)\cup \{\bs(x-):x\in D^-\}\cup \{\bs(x+): x\in D^+\}.  
\]
By means of $\widehat{\fm}=\fm\circ \bs^{-1}$ and \eqref{eq:35}, we can obtain the first and second identities in \eqref{eq:31}.  If $l=c_n$ for some $n$,  then (DK) implies that $d_n\notin D$.  Particularly,  
\[
	\widehat{I}= \bs(I\setminus [c_n,d_n])\cup \{\bs(x-):x\in D^-\}\cup \{\bs(x+): x\in D^+\}.  
\]
Note that $f(x)=0$ for $f\in \sF$ and $x\in [c_n,d_n]$ and $\widehat{f}(\widehat{l})=0$ for $\widehat{f}\in \widehat{\sF}$.  These facts,  together with $\widehat{\fm}=\fm\circ \bs^{-1}$ and \eqref{eq:35},  yield the first and second identities in \eqref{eq:31}. 
That completes the proof.
\end{proof}

\begin{remark}\label{RM38}
%In view of Remark~\ref{RM36},  $\widehat{l}=\widehat{a}_k$ for some $k\geq 1$,  equivalently $\widehat{l}$ is isolated in $\overline{\bs([l,r])}$,  if and only if $c_n=l, d_n\notin D$ for some $k\geq 1$.  Meanwhile $\widehat{l}$ is reflecting (so not absorbing) and $\widehat{l}\in \widehat{I}$.    The right endpoint $\widehat{r}$ can be argued analogously. 
Without assuming (DK),  it might occur that for some $n$,
\[
	l=c_n<d_n\in D,\quad \fm(l+)+\fm(\{l\})=\infty.
\]
Then $\widehat{l}$ is isolated in $\overline{\bs([l,r])}$ with $\widehat{\fm}(\{\widehat{l}\})=\infty$ and $\widehat{l}=\widehat{a}_k$ for some $k$.  This endpoint becomes an absorbing point and we must have $\widehat{f}(\widehat{l})=0$ for $\widehat{f}\in \widehat{\sF}$.  As a consequence,  $\widehat{\sE}(\widehat{f},\widehat{f})$ contains the killing part $\widehat{f}(\widehat{b}_k)^2/(2|\widehat{b}_k-\widehat{a}_k|)$.  In other words,  the hypothesis (DK) excludes killing part from the desirable Dirichlet form. 
\end{remark}

\subsection{Canonical regularized Markov process}\label{SEC33}

It is well known that every regular Dirichlet form corresponds to a (unique) symmetric Hunt process; see,  e.g., \cite{FOT11}.  Denote by $\widehat{X}=(\widehat{X}_t)_{t\geq 0}$ the  $\widehat{\fm}$-symmetric Hunt process on $\widehat{I}$ associated to $(\widehat{\sE}, \widehat{\sF})$.  %Note that $\widehat{X}$ is equivalent to $\left(\bs^*(X^*_t)\right)_{t\geq 0}$ in the sense that their transition functions coincide.  
We call it the \emph{canonical regularized Markov process associated to} $(I,\bs,\fm)$.  
In this subsection we will show that $\widehat{X}$ is a time-changed Brownian motion with speed measure $\widehat{\fm}$.  

Let $\widehat{J}:=\langle \widehat{l},\widehat{r}\rangle$ be the interval as in the proof of Theorem~\ref{THM6},  i.e.  $\widehat{l}\in \widehat{J}$ ($\widehat{r}\in \widehat{J}$) if and only if $\widehat{l}\in \widehat{I}$ ($\widehat{r}\in \widehat{I}$).  Denote by $\widehat B=(\widehat B_t)_{t\geq 0}$ the Brownian motion on $\widehat{J}$ which is absorbing at each finite open endpoint and reflecting at each finite closed endpoint.  The life time of $\widehat{B}$ is denoted by $\widehat{\zeta}$.  %For example,  when $\widehat{J}=[0,1)$,  $\widehat B$ is a Brownian motion on $[0,1)$, which is reflecting at $0$ while absorbing at $1$,  and its lifetime is identified with the first time when $\widehat{B}$ hits $1$. 
The associated Dirichlet form of $\widehat B$ on $L^2(\widehat{J})$ is 
\[
\begin{aligned}
	&H^1_0(\widehat{J}):=\bigg\{\widehat f\in L^2(\widehat{J}): \widehat f\text{ is absolutely continuous on } (\widehat{l},\widehat{r}) \text{ and }\\
	&\qquad \qquad\qquad \int_{\widehat{J}}\widehat f'(\widehat x)^2d\widehat x<\infty,  \widehat f(\widehat j)=0\text{ if }\widehat j\notin \widehat{J}\text{ is finite for }\widehat j=\widehat{l} \text{ or }\widehat{r}\},\\
	&\frac{1}{2}\mathbf{D}(\widehat f,\widehat g):=\frac{1}{2}\int_{\widehat{J}}\widehat f'(\widehat x)\widehat g'(\widehat x)d\widehat x,\quad \widehat f,\widehat g\in H^1_0(\widehat{J}).  
\end{aligned}\]
Since viewed as a zero extension to $\widehat{J}$ is Radon on $\widehat{J}$,  $\widehat{\fm}$ is smooth with respect to $(\frac{1}{2}\mathbf{D}, H^1_0(\widehat{J}))$.  Clearly the quasi support of $\widehat{\fm}$ is identified with its topological support $\widehat{I}$.   Denote the PCAF of $\widehat{\fm}$ with respect to $\widehat{B}$ by $\widehat{A}=(\widehat{A}_t)_{t\geq 0}$.  Set 
\[
\widehat \tau_t:=\left\lbrace 
	\begin{aligned}
		&\inf\{s: \widehat{A}_s>t\},\quad t<\widehat{A}_{\widehat{\zeta}-},\\
		&\infty,\qquad\qquad\qquad\;\; t\geq \widehat{A}_{\widehat{\zeta}-}  
	\end{aligned}
\right.
\]
and $\check{X}_t:=\widehat{B}_{\widehat\tau_t}, \check{\zeta}:=\widehat{A}_{\widehat{\zeta}-}$.  
Then $\check{X}$ is a \emph{right process} on $\widehat{I}$ with lifetime $\check{\zeta}$,  called the \emph{time-changed Brownian motion} with speed measure $\widehat \fm$;  see,  e.g.,  \cite[Theorem~A.3.11]{CF12}.  

The following result,  together with \cite[Corollary~5.2.10]{CF12},  shows that the canonical regularized Markov process $\widehat{X}$ is identified with $\check{X}$.  This argument also gives an alternative proof for the regularity of $(\widehat{\sE},\widehat{\sF})$.  
The terminologies concerning trace Dirichlet forms are referred to in,  e.g.,  \cite[\S5.2]{CF12}. 

% Denote the $\fm^*$-symmetric Hunt process on $I^*$ associated with $(\sE^*,\sF^*)$ by $X^*$,  also called the \emph{Markov process associated with} $(I,\bs,\fm)$.  Then $(\widehat{\sE},\widehat{\sF})$ is associated with the $\widehat{\fm}$-symmetric Hunt process $(\widehat{X}_t)_{t\geq 0}:=\left(\bs^*(X^*_t)\right)_{t\geq 0}$ on $\widehat{I}$,  which is called the \emph{image Markov process associated with} $(I,\bs,\fm)$.  It is on the natural scale $\widehat{\bs}$,  and 

\begin{theorem}\label{THM19}
$(\widehat{\sE},\widehat{\sF})$ is the trace Dirichlet form of $(\frac{1}{2}\mathbf{D}, H^1_0(\widehat{J}))$ on $L^2(\widehat{I}, \widehat{\fm})$.  
\end{theorem}
\begin{proof}
Denote by $(\check{\sE},\check{\sF})$ the trace Dirichlet form of $(\frac{1}{2}\mathbf{D}, H^1_0(\widehat{J}))$ on $L^2(\widehat{I}, \widehat{\fm})$.  Let $H^1_e(\widehat{J})$ be the extended Dirichlet space of $(\frac{1}{2}\mathbf{D}, H^1_0(\widehat{J}))$,  i.e. 
\begin{equation}\label{eq:320-2}
	\begin{aligned}
	&H^1_e(\widehat{J}):=\bigg\{\widehat f\text{ on }\widehat{J}:  \widehat f\text{ is absolutely continuous on }(\widehat{l},\widehat{r}) \text{ and }\\
	&\qquad \qquad\qquad \int_{\widehat{J}}\widehat f'(x)^2dx<\infty,  \widehat f(\widehat j)=0\text{ if }\widehat j\notin \widehat{J}\text{ is finite for }\widehat j=\widehat{l} \text{ or }\widehat{r}\}; 
	\end{aligned}
\end{equation}
see,  e.g.,  \cite[(3.5.10)]{CF12}.
We only need to prove
\begin{equation}\label{eq:17}
	H^1_e(\widehat{J})|_{\widehat{I}}\cap L^2(\widehat{I},\widehat{\fm})=\widehat{\sF}
\end{equation}
and for $\widehat{\varphi}\in \widehat{\sF}$,  $\check{\sE}(\widehat{\varphi},\widehat{\varphi})=\widehat{\sE}(\widehat{\varphi}, \widehat{\varphi})$.  The identity \eqref{eq:17} can be straightforwardly verified by means of Lemma~\ref{LM39} and \eqref{eq:321}.  The expression of $\check{\sE}(\widehat{\varphi},\widehat{\varphi})$ for $\widehat{\varphi}\in H^1_e(\widehat{J})|_{\widehat{I}}\cap L^2(\widehat{I},\widehat{\fm})$ can be obtained by mimicking the proof of \cite[Theorem~2.1]{LY17}.  It is identified with $\widehat{\sE}(\widehat{\varphi},\widehat{\varphi})$ expressed as \eqref{eq:321}.  That eventually completes the proof. 
%To do this,  note that an absorbing endpoint $\widehat{j}$ of $\widehat{I}$ must be finite and $\widehat{j}\notin \widehat{J}$.  To the contrary,  
%take $f\in H^1_e(\widehat{J})$.  We have
%\[
%	\sum_{k\geq 1}\frac{\left({f}(\widehat{a}_k)-{f}(\widehat{b}_k)\right)^2}{|\widehat{b}_k-\widehat{a}_k|}\leq \sum_{k\geq 1}\frac{|\widehat{b}_k-\widehat{a}_k|\cdot \int_{\widehat{a}_k}^{\widehat{b}_k} f'(x)^2dx}{|\widehat{b}_k-\widehat{a}_k|}=\int_{\widehat{J}\setminus \widehat{I}}f'(x)^2dx<\infty.
%\]
%Hence $f|_{\widehat{I}}\in \widehat{\sS}_0$.  To the contrary,  let $f\in \sS_0$.  Extend $f$ to a function $\tilde{f}$ on $\widehat{J}$ in the following way:
%\[
%\tilde{f}|_{\widehat{I}}:=f,\quad \tilde{f}(x):=f(\widehat{a}_k)+\frac{f(\widehat{b}_k)-f(\widehat{a}_k)}{\widehat{b}_k-\widehat{a}_k}\cdot (x-\widehat{a}_k),\; x\in (\widehat{a}_k,\widehat{b}_k), k\geq 1.  
%\]	
%It is easy to verify that $\tilde{f}$ is absolutely continuous on $\widehat{J}$ and 
%\[
%\int_{\widehat{J}}\tilde f'(x)^2dx=\int_{\widehat{I}}f'(x)^2dx+\sum_{k\geq 1}\frac{\left({f}(\widehat{a}_k)-{f}(\widehat{b}_k)\right)^2}{|\widehat{b}_k-\widehat{a}_k|}=2\widehat{\sE}(f,f)<\infty.  
%\]
%This yields $f\in H^1_e(\widehat{J})|_{\widehat{I}}$.  Hence \eqref{eq:17} is verified.  The identity $\check{\sE}(\widehat{\varphi},\widehat{\varphi})=\widehat{\sE}(\widehat{\varphi}, \widehat{\varphi})$ for $\widehat{\varphi}\in \widehat{\sS}_0$ can be obtained as in the proof of 
\end{proof}

As obtained in Theorem~\ref{THM6},  $(\widehat{\sE},\widehat{\sF})$ is irreducible.  By virtue of Theorem~\ref{THM19},  we give a criterion for  global property of $(\widehat{\sE},\widehat{\sF})$ or $\widehat{X}$. 

\begin{corollary}\label{COR310}
The following hold:
\begin{itemize}
\item[\rm (1)] $(\widehat{\sE},\widehat{\sF})$ is transient,  if and only if either $\widehat{l}\in \bR\setminus \widehat{I}$ or $\widehat{r}\in \bR\setminus \widehat{I}$.  This is also equivalent to that either $l$ or $r$ is approachable but not reflecting.  Otherwise $(\widehat{\sE},\widehat{\sF})$ is recurrent.
\item[\rm (2)] Every singleton contained in $\widehat{I}$ is of positive capacity with respect to $\widehat{\sE}$.  Particularly,  $\widehat{X}$ is pointwisely irreducible in the sense that 
\[
\widehat{\mathbf{P}}_{\widehat{x}}(\widehat{\sigma}_{\widehat{y}}<\infty)>0
\]
for any $\widehat{x},\widehat{y}\in \widehat{I}$,  where $\widehat{\mathbf{P}}_{\widehat{x}}$ is the probability measure on the sample space of $\widehat{X}$ starting from $\widehat{x}$ and $\widehat{\sigma}_{\widehat{y}}:=\inf\{t>0: \widehat{X}_t=\widehat y\}$.  
\end{itemize}
\end{corollary}
\begin{proof}
In view of \cite[Theorem~5.2.5]{CF12},  $(\widehat{\sE},  \widehat{\sF})$ is transient if and only if so is $(\frac{1}{2}\mathbf{D}, H^1_0(\widehat{J}))$.  This,  together with \cite[Theorem~2.2.11]{CF12},  yields the first condition equivalent to the transience of $(\widehat{\sE},  \widehat{\sF})$.  The second equivalent condition is obvious.  Another assertion is the consequence of \cite[Theorems~3.5.6~(1) and 5.2.8~(2)]{CF12}.  That completes the proof. 
\end{proof}

\subsection{Homeomorphisms between regular representations}\label{SEC34}

We turn to show that all regular representations of $(I,\fm,\sE,\sF)$ are essentially homoemorphic.  In other words,  a Markov process corresponding to certain regular representation must be a homeomorphic image of $\widehat{X}$.  The lemma below is useful for obtaining this result.  

\begin{lemma}\label{LM310}
Let $\{\widehat{F}_n: n\geq 1\}$ be an $\widehat{\sE}$-nest and $\widehat{K}$ be a compact subset of $\widehat{I}$.  Then $\widehat{K}\subset \widehat{F}_n$ for some $n\geq 1$.  
\end{lemma}
\begin{proof}
 %suffices to show that $K\subset \widehat{F}_n$ for some $n$.  
We first prove that for any $\widehat x\in \widehat K$,  there exists $\varepsilon>0$ such that
\begin{equation}\label{eq:76}
	(\widehat x-\varepsilon, \widehat x+\varepsilon)\cap (\widehat I\setminus \widehat{F}_n)=\emptyset,\quad \text{for some }n\geq 1.  
\end{equation}
Argue by contradiction and take $\widehat x\in \widehat K$ such that $(\widehat x-\varepsilon,\widehat x+\varepsilon)\cap (\widehat I\setminus \widehat{F}_n)\neq \emptyset$ for any $\varepsilon>0$ and $n\geq 1$.  Particularly,  there exists a sequence $\widehat x_n\in \widehat I\setminus \widehat{F}_n$ such that $\widehat x_n\rightarrow \widehat x$.   Take $\widehat{f}\in \widehat{\sF}_{\widehat{F}_k}:=\{f\in \widehat{\sF}: f=0\text{ on }\widehat I\setminus \widehat{F}_k\}$ for some $k$.  Clearly $\widehat{f}|_{\widehat I\setminus \widehat{F}_n}\equiv 0$ for $n\geq k$.  Since every function in $\widehat{\sF}$ is continuous on $\widehat{I}$,  it follows that $\widehat{f}(\widehat x)=\lim_{k<n\rightarrow \infty}\widehat{f}(\widehat x_n)=0$.  Particularly,  
\[
	\cup_{k\geq 1}\widehat{\sF}_{\widehat{F}_k}\subset \{\widehat{f}\in \widehat{\sF}: \widehat{f}(\widehat x)=0\}.  
\]
The family on the left hand side is $\widehat{\sE}_1$-dense in $\widehat{\sF}$ while the right one is not.  This leads to a contradiction.  As a result it follows from \eqref{eq:76} that for any $\widehat x\in \widehat K$,  there exists $\varepsilon>0$ such that $(\widehat x-\varepsilon, \widehat x+\varepsilon)\cap \widehat I\subset \widehat{F}_n$ for some $n$.  Using the compactness of $\widehat K$,  we can obtain that $\widehat K\subset \widehat{F}_n$ for some $n$.  That completes the proof. 
\end{proof}

Before stating the result,  we prepare some notations and terminologies.  
Let $(\sE^1,\sF^1)$ be a Dirichlet form on $L^2(E_1,\fm_1)$.  Take another measurable space $(E_2,\mathcal{B}(E_2))$ and a measurable map $j: (E_1,\mathcal{B}(E_1))\rightarrow (E_2,\mathcal{B}(E_2))$.  Define $\fm_2:=\fm_1\circ j^{-1}$,  the image measure of $\fm_1$ under $j$.  Then 
\[
	j_*: L^2(E_2,\fm_2)\rightarrow L^2(E_1,\fm_1),\quad f\mapsto j_*f:=f\circ j
	\]
 is an isometry.  Define $\sF^2:=\{f\in L^2(E_2,\fm_2): j_*f\in \sF^1\}$ and 
 \[
 	\sE^2(f,g):=\sE^1(j_*f, j_*g),\quad f,g\in \sF^1. 
 \]
 If $j_*$ maps $L^2(E_2,\fm_2)$ onto $L^2(E_1,\fm_1)$,  then $(\sE^2,\sF^2)$ is a Dirichlet form on $L^2(E_2,\fm_2)$,  which is called the \emph{the image Dirichlet form} of $(\sE^1,\sF^1)$ under $j$.  
 Particularly,  if both $E_1$ and $E_2$ are locally compact separable metric spaces and $j$ is an a.e. homeomorphism,  i.e.  there is an $\fm_1$-negligible set $N_1$ and an $\fm_2$-negligible set $N_2$ such that $j: E_1\setminus N_1\rightarrow E_2\setminus N_2$ is a homeomorphism,   then $j_*$ is surjective.  

%Clearly a homeomorphic image of $X^*$ also corresponds to a regular representation of $(I,\fm,\sE,\sF)$.  The following result leads to the contrary: The Markov process corresponding to a regular representation must be a homeomorphic image of $X^*$.  Note that the concept of \emph{image Dirichlet form} is referred to in \cite[\S1.4]{CF12}. 

\begin{theorem}\label{THM311}
Let $(I',\fm',\sE',\sF')$ be a regular representation of $(I,\fm, \sE,\sF)$.  Then there exists a unqiue $\sE'$-polar set $N'\subset I'$ and a unique homeomorphism $j': \widehat{I} \rightarrow I'\setminus N'$ such that $(\sE',\sF')$ is the image Dirichlet form of $(\widehat\sE,\widehat \sF)$ under $j'$.  
\end{theorem}
\begin{proof}
In view of \cite[Theorem~A.4.9]{FOT11},  there exists a regular representation $(\tilde{I}, \tilde{\fm}, \tilde{\sE},\tilde{\sF})$ such that both $(I',\fm',\sE',\sF')$ and $(\widehat{I},\widehat{\fm},\widehat{\sE},\widehat{\sF})$ are equivalent to it by isomorphisms $\Phi'$ and $\widehat{\Phi}$, and 
\[
	\Phi'\left(\sF'\cap C_\infty(I') \right)\subset \tilde{\sF}\cap C_\infty(\tilde{I}),\quad \widehat\Phi(\widehat\sF\cap C_\infty(\widehat I) )\subset \tilde{\sF}\cap C_\infty(\tilde{I}).
\]
Applying \cite[Lemma~A.4.8]{FOT11} to $\widehat{\Phi}$ and repeating its proof,  we can obtain a continuous map $\widehat{\gamma}: \tilde{I}\rightarrow \widehat{I}$,  an $\tilde{\sE}$-nest $\{\tilde{F}^1_n: n\geq 1\}$ and an $\widehat{\sE}$-nest $\{\widehat{F}_n\}$ such that $f\circ \widehat{\gamma}\in C_\infty(\tilde{I})$ for any $f\in C_\infty(\widehat{I})$ and
\begin{equation}\label{eq:319}
	\widehat{\gamma}_n:=\widehat{\gamma}|_{\tilde{F}^1_n}: \tilde{F}^1_n\rightarrow \widehat{F}_n,\quad n\geq 1
\end{equation}
are homeomorphisms.  In addition,  $(\widehat{\sE},\widehat{\sF})$ is the image Dirichlet form of $(\tilde{\sE},\tilde{\sF})$ under $\widehat{\gamma}$.  Analogously there is a continuous map $\gamma': \tilde{I}\rightarrow I'$,  an $\tilde{\sE}$-nest $\{\tilde{F}^2_n: n\geq 1\}$ and an $\tilde{\sE}'$-nest $\{F'_n: n\geq 1\}$ such that $f\circ \gamma'\in C_\infty(\tilde{I})$ for any $f\in C_\infty(I')$ and
\begin{equation}\label{eq:320-3}
	\gamma'_n:=\gamma'|_{\tilde{F}^2_n}: \tilde{F}^2_n \rightarrow F'_n,\quad n\geq 1
\end{equation}
are homeomorphisms.  In addition,  $(\sE',\sF')$ is the image Dirichlet form of $(\tilde{\sE},\tilde{\sF})$ under $\gamma'$.  Without loss of generality we may and do assume $\tilde{F}^1_n=\tilde{F}^2_n=:\tilde{F}_n$.  (Otherwise we can replace $\tilde{F}^1_n$ and $\widehat{F}_n$ by $\tilde{F}^1_n\cap \tilde{F}^2_n$ and  $\widehat{F}_n \cap \widehat{\gamma}(\tilde{F}^2_n)$ in \eqref{eq:319}.  The maps in \eqref{eq:320-3} can be treated similarly.) On account of Corollary~\ref{COR310}~(2),  we have 
\[
	\widehat{I}=\cup_{n\geq 1}\widehat{F}_n.
\]
Set $I'_0:=\cup_{n\geq 1}F'_n$ and $\tilde{I}_0:=\cup_{n\geq 1}\tilde{F}_n$.  
%and every singleton contained in  (resp.  ) is of positive $\tilde{\sE}$-capacity (resp.  $\sE'$-capacity).  

Secondly,  we assert that for any compact subset $K'$ of $I'$,  it holds that $K'\cap I'_0\subset F'_n$ for some $n\geq 1$.  %Analogical property also holds for a compact subset $\tilde{K}$ of $\tilde{I}$.  
To do this,  take $f\in C_\infty(I')$ such that $f=1$ on $K'$,  and set $\tilde{K}:=\gamma'^{-1}(K')$.  Since $\gamma'$ is continuous,  $\tilde{K}$ is closed in $\tilde{I}$.  In addition,  
\[
	\tilde{K}\subset \{\tilde{x}\in \tilde{I}: f\circ \gamma'(\tilde{x})=1\}
\]
and $f\circ \gamma'\in C_\infty(\tilde{I})$ yields that the right hand side is a subset of a compact set in $\tilde{I}$.  Particularly $\tilde{K}$ is compact in $\tilde{I}$.  It follows from the continuity of $\widehat{\gamma}$ that $\widehat{K}:=\widehat{\gamma}(\tilde{K})$ is compact in $\widehat{I}$.  Applying Lemma~\ref{LM310} to $\widehat{K}$ and using homeomorphisms \eqref{eq:319} and \eqref{eq:320-3},  we get that $\widehat{K}\subset \widehat{F}_n$ for some $n$ and thus 
\[
	K'\cap I'_0= \gamma'(\tilde{K}\cap \tilde{I}_0)\subset \gamma'(\widehat \gamma^{-1}(\widehat{K})\cap \tilde{I}_0)\subset \gamma'_n(\widehat{\gamma}^{-1}_n(\widehat{F}_n))=F'_n.  
\]  
%Analogical assertion for $\tilde{K}$ can be derived similarly.  

Denote by $\widehat{q}$ the inverses of $\widehat{\gamma}|_{\tilde{I}_0}$.  We show that $\widehat{q}$ is continuous on $\widehat{I}$,  so that $\widehat{q}: \widehat{I}\rightarrow \tilde{I}_0$
is a homeomorphism and $(\tilde{\sE},\tilde{\sF})$ is the image Dirichlet form of $(\widehat{\sE},\widehat{\sF})$ under $\widehat{q}$.  To this end,  take an arbitrary precompact open subset  $\widehat U$ of $\widehat{I}$.  Since $\widehat{U}\subset \widehat{F}_n$ for some $n$,  it follows that $\widehat q|_{\widehat{U}}=\widehat{\gamma}^{-1}_n|_{\widehat{U}}$ is continuous.  Consequently $\widehat{q}$ is continuous on $\widehat{I}$.  %Concerning $q'$,  we take an arbitrary precompact open subset $U'$ of $I'_0$.  Note that the closure $\bar{U}'$ of $U'$ is compact 

%We first assert that given an $\widehat{\sE}$-nest  and a compact set $\widehat{K}\subset \widehat{I}$,  it holds that $\widehat{K}\subset \widehat{I}$ for some $n\geq 1$.  

Set $N':=I'\setminus I'_0$ and 
\begin{equation}\label{eq:321-2}
	j': \widehat{I}\rightarrow I'_0,\quad \widehat{x}\mapsto \gamma'(\widehat{q}(\widehat{x})). 
\end{equation}
Then $j'$ is a continuous bijection and its restriction $j'|_{\widehat{F}_n}: \widehat{F}_n\rightarrow F'_n$ is a homoemorphism.  We prove that $j'$ is a local homeomorphism.  Indeed,  let $\widehat{x}\in\widehat{I}$ and $x':=j'(\widehat{x})$.  Since $I'$ is locally compact,  we take a precompact open set $V$ in $I'$ with $x'\in V$ and set $V':=V\cap I'_0$.  Then $V'$ is an open neighbourhood of $x'$ in $I'_0$ and the second step yields that 
\[
	V'\subset \overline{V}\cap I'_0\subset F'_n
\]
for some $n$,  where $\overline{V}$ is the closure of $V$ in $I'$.  Since $j'$ is continuous and $j'|_{\widehat{F}_n}$ is a homoemorphism,  it follows that $\widehat U:=j'^{-1}(V')\subset \widehat{F}_n$ is an open neighbourhood of $\widehat{x}$ in $\widehat{I}$ and $j'|_{\widehat{U}}: \widehat{U}\rightarrow V'$ is a homeomorphism.  Therefore \eqref{eq:321-2} is a local homeomorphism.  Note that $j'$ is also bijective.  Eventually we conclude that $j'$ is a homeomorphism.  

Making use of that both $j'$ and $\widehat{q}$ are homeomorphisms,  one can find that $\gamma'|_{\tilde{I}_0}: \tilde{I}_0\rightarrow I'_0$ is also a homeomorphism and hence can easily verify that $(\sE',\sF')$ is the image Dirichlet form of $(\widehat{\sE},\widehat{\sF})$ under $j'$.  

Finally we argue the uniqueness of $(N',j')$.  Take another pair $(N'_1, j'_1)$ with the same properties.  In view of Corollary~\ref{COR310}~(2),  every singleton contained in $\widehat{I}$ is not $\widehat{\sE}$-polar.  Since $j'_1$ is a homeomorphism,  it follows that every singleton contained in $I'\setminus N'_1$ is not $\sE'$-polar.  Consequently $I'\setminus N'_1\subset I'\setminus N'$ because $N'$ is $\sE'$-polar.  The contrary $I'\setminus N'\subset I'\setminus N'_1$ also holds true by a similar argument.  Therefore $N'=N'_1$.  The identity $j'=j'_1$ is obvious because $j'(\widehat{X})$ and $j'_1(\widehat{X})$ are the identical Markov process associated to $(\sE',\sF')$.  That completes the proof.  
%By making use of \cite[A.4.2]{FOT11},  we can obtain a quasi-homeomorphism $j$ between $(\sE^*,\sF^*)$ and $(\sE',\sF')$.  Particularly,  $(\sE,\sF')$ is the image Dirichle form of $(\sE^*,\sF^*)$ under $j$,  and there is an $\sE^*$-nest $\{F^*_n: n\geq 1\}$ and an $\sE'$-nest $\{F'_n: n\geq 1\}$ such that $j|_{F^*_n}$ is a homeomorphism from $F^*_n$ to $F'_n$.  In view of Corollary~\ref{COR310}~(2),  we have $I^*=\cup_{n\geq 1}F^*_n$.  Set $N':=I'\setminus \cup_{n\geq 1}F'_n$,  which is certainly an $\sE'$-polar set.  It suffices to show that 
%\[
%	j: I^*\rightarrow I'_0:= I'\setminus N'= \cup_{n\geq 1} F'_n
%\]
%is a homeomorphism.  This will be accomplished in several steps as follows. 
%Firstly we note that every singleton contained in $I'_0$ is not $\sE'$-polar due to \cite[Exercise~1.4.2~(ii)]{CF12}. 
%Denote by $d'$ the metric inducing the restricted topology of $I'_0$.  Secondly we assert that for any $x\in I'_0$,  there are some $\varepsilon>0$ and $n\geq 1$ such that 
% \begin{equation}\label{eq:316-2}
% 	\{y\in I'_0: d'(y,  x)<\varepsilon\}\cap \left(I'_0\setminus F'_n \right)=\emptyset.  
% \end{equation}
% Argue by contradiction and suppose that $x\in I'_0$ does not satisfy \eqref{eq:316-2} for any $\varepsilon>0$ and $n\geq 1$.  Then there is a sequence $x_n\in F'_n$ such that $d'(x_n, x)\rightarrow 0$ as $n\rightarrow \infty$.   Take $f\in \sF'_{F'_k}:=\{g\in \sF': g=0\text{ on }I'_0\setminus F'_k\}$ for some $k$.  Clearly $f|_{I'_0\setminus F'_n}=0$ for any $n\geq k$.  Thus $f(x)=$
\end{proof}

The following corollary is immediate from this theorem.

\begin{corollary}
Let $(I_i,\fm_i, \sE^i,\sF^i)$ be the regular representations of $(I,\fm,  \sE,\sF)$ for $i=1,2$.  Then there are $\sE^i$-polar sets $N_i\subset I_i$ for $i=1,2$ and a homeomorphism $j: I_1\setminus N_1\rightarrow I_2\setminus N_2$ such that $(\sE^2,\sF^2)$ is the image Dirichlet form of $(\sE^1,\sF^1)$ under $j$.  
\end{corollary}

Denote by $X'$ the $\fm'$-symmetric Hunt process associated to $(\sE',\sF')$.  This process would be called a \emph{regularized Markov process associated to} $(I,\bs,\fm)$ if no confusions caused.  In addition, we call $N'$ obtained in Theorem~\ref{THM311} the \emph{essentially exceptional set} of $X'$ or $(I',\fm',\sE',\sF')$,  and $\bs':=j'^{-1}$,  the inverse of $j'$,  the \emph{scale function} of  $X'$ or $(I',\fm',\sE',\sF')$.  Particularly,  the essentially exceptional set of $\widehat{X}$ is empty,   and its scale function is $\widehat{\bs}$.  We wish to state emphatically that in general, the essentially exceptional set is not necessarily empty.  

\begin{example}
Consider that $I=[0,1)$,  $\fm$ is the Lebesgue measure on $[0,1)$ and $\bs(x)$ is continuous and strictly increasing on $I$ such that $\bs(0)=0$ and $\bs(1)=\infty$.  Then $(I,\fm, \sE,\sF)$ corresponds to the regular diffusion on $I$ with scale function $\bs$,  speed measure $\fm$ and no killing inside.  

Making use of \cite[Theorem~2.1]{LY19},  one can get that $(\sE,\sF)$ is regular on not only $L^2(I,\fm)$ but also $L^2([0,1],\fm)$.  Particularly,  $([0,1],\fm,\sE,\sF)$ is a regular representation of $(I,\fm,  \sE,\sF)$,  and its essentially exceptional set is $\{1\}$.  
\end{example}

\section{Unregularized Markov process and Ray-Knight compactification}\label{SEC31}

In this section we will apply two ``inverse" maps of $\bs$ to the canonical regular representation $(\widehat{I},\widehat{\bs},\widehat{\sE},\widehat{\sF})$,  so that two different Markov processes are obtained.  The first one is a homeomorphism,  and the resulting Markov process is still a regularized Markov process $X^*$ associated to $(I,\bs,\fm)$ (see \S\ref{SEC34}).  While the second maps $\bs(x-)$ or $\bs(x+)$ also to $x$ whenever $x\in D^-$ or $x\in D^+$,  and eventually a Markov process $\dot X$,  for which the strong Markov property may fail, is obtained.  Their importance lies on that the Dirichlet form of $\dot X$ is exactly $(\sE,\sF)$ and that $X^*$ is the Ray-Knight compactification of $\dot X$.  

\subsection{Regularized Markov process}\label{SEC41}

To find out a homeomorphic ``inverse" of the map
\[
	\bs: I\rightarrow \overline{\bs([l,r])},
\]
we meet two difficulties.  Firstly,  the second and third components of $\overline{\bs([l,r])}$ expressed as \eqref{eq:34} admit no corresponding parts contained in $I$.  Secondly,  the intervals in \eqref{eq:14} are mapped to a single point under $\bs$.  The natural way to tackle them is to change $I$ into another space $I^*$ with the help of following two transformations:
\begin{itemize}
\item[(1)] The first one,  due to \cite{S79} and called \emph{scale completion},  makes the completion $\bar{I}^\rho$ of $I$ with respect to the metric
\[
	\rho(x,y):=|\tan^{-1} x-\tan^{-1}y|+|\tan^{-1}\bs(x)-\tan^{-1}\bs(y)|.  
\] 
This transformation divides $x\in D^0$ into $\{x-,x,x+\}$ and $x\in D \setminus D^0$ into $\{x-,x+\}$.  Note that for a sequence $I\setminus D\ni  x_n\uparrow x\in D$ (resp.  $I\setminus D\ni x_n\downarrow x\in D$) in $I$,  we have $\rho(x_n,  x-)\rightarrow 0$ (resp.  $\rho(x_n, x+)\rightarrow 0$). 
\item[(2)] The second is \emph{darning}.  Let $\overline{(c_n,d_n)}^\rho$ be the closure of $(c_n,d_n)$ in $\bar{I}^\rho$,  where $(c_n,d_n)$ appears in \eqref{eq:13}.  The transformation of darning collapses each $\overline{(c_n,d_n)}^\rho$ into an abstract point $p^*_n$ and the neighbourhoods of $p^*_n$ are determined by those of $\overline{(c_n,d_n)}^\rho$ in $\bar{I}^\rho$.  We refer readers to page 347 of \cite{CF12} for more details about this operation.  Denote by $\bar{I}^{\rho,*}$ the space obtained by darning $\bar{I}^\rho$.  
\end{itemize}

%The crucial step to obtain a regular representation of $(I,\fm,\sE,\sF)$ is to utilize two transformations on the state space: 
%In advance of treating $(\sE,\sF)$,  we will transform $(I,\bs,\fm)$ into another triple $(I^*, \bs^*, \fm^*)$,  called the \emph{regularization} of $(I,\bs,\fm)$.  
%,  so that a regular Dirichlet form $(\sE^*,\sF^*)$ on $L^2(I^*,\fm^*)$ is obtained by transforming $(\sE,\sF)$ defined as \eqref{eq:25} accordingly.  
%\subsubsection{Regularized state space}\label{SEC331}
%Two transformations are utilized for regularizing the state space: 
 %The partial ordering ``$<$'' on $\bar{I}^\rho$ can be easily constructed by letting $x-<x<x+$ for $x\in D^0$ and $x-<x+$ for $x\in D\setminus D^0$.

%The second step is the darning transformation on $\bar{I}^\rho$.  The right endpoint of $\overline{(c_n,d_n)}^\rho$ is $d_n-$ if $d_n\in D$ and $d_n$ if $d_n\in I\setminus D$.  The left endpoint is analogical.    and the new topological structure near $p^*_n$ can be described in the same way as in page 347 of \cite{CF12}.  Roughly speaking,   .  %These two transformations overcome the two difficulties (a) and (b) mentioned in the beginning of \S\ref{SEC14}. 

\begin{example}\label{EXA31}
This example is to explain the above two transformations.  Consider
\[
	I=[0,3], \quad \bs(x)=\left\lbrace\begin{aligned}
	&x,\quad 0\leq x< 1,\\
	&1,\quad 1\leq x< 2,\\
	&x,\quad 2\leq x\leq 3.  
	\end{aligned}  \right.  
\]
Then $U=(c_1,d_1)=(1,2)$ and $\bs$ has only one discontinuous point $2$.  The transformation of scale completion divides $2$ into $\{2-, 2+\}$ and $$\bar{I}^\rho=[0,2-]\cup [2+, 3],$$  where $[0,2-]$ (resp.  $[2+,3]$) is homoemorphic to the usual interval $[0,2]$ (resp.  $[2,3])$ but $2-$ and $2+$ are distinct points in $\bar{I}^\rho$.  The closure $\overline{(c_1,d_1)}^\rho$ of $(c_1,d_1)$ is $[1,2-]$ and the darning transformation collapses it into an abstract point $p^*_1$,  which can be viewed as the usual point $1$.  In other words,  $\bar{I}^{\rho,*}$ may be treated as $[0,1]\cup [2+,3]$.  

We should point out that the abstract point $p^*_n$ might not always be understood as a usual point.  For example,  consider $I=[0,1]$ and that $\bs$ is the standard Cantor function on $[0,1]$,  i.e.  $\bs(x)=\int_0^x 1_{K^c}(y)dy$ where $K\subset [0,1]$ is the standard Cantor set.  Then $\bar{I}^\rho=[0,1]$,  and darning transformation collapses each open interval in the decomposition of $[0,1]\setminus K$  into an abstract point $p^*_n$.  If all $p^*_n$ are viewed as usual points of zero Lebesgue measure,  $\bar{I}^{\rho,*}$,  not a singleton,  must be identified with a negligible interval.  This is incomprehensible. 
\end{example}

By means of the transformations of scale completion and darning,  one can define a censored map
$\bs^*: \bar{I}^{\rho,*}\rightarrow [-\infty, \infty]$ as follows:
\begin{equation}\label{eq:32}
	\bs^*(x^*):=\left\lbrace
	\begin{aligned}
	&\bs(x),\quad \quad\; x^*=x\in \left(I\setminus (D\setminus D^0)\right) \cap \bar{I}^{\rho,*},  \\
	&\bs(x\pm),\quad\;\; x^*=x\pm\in \bar{I}^{\rho,*}\text{ with }x\in D,\\
	&\bs(c_n+),\quad\, x^*=p^*_n\text{ for each } n\geq 1.
	\end{aligned}
	\right. 
\end{equation}
It is straightforward to verify the following.

\begin{lemma}\label{LM32}
The map $\bs^*$ is a homeomorphism between $\bar{I}^{\rho,*}$ and $\overline{\bs([l,r])}$.  
\end{lemma}

Let $\br^*$ be the inverse of this homeomorphism $\bs^*$.  Set
\[
I^*:=\br^*(\widehat{I}), \quad \fm^*:=\widehat{\fm}\circ \bs^*
\]
and 
\[
\begin{aligned}
	&\sF^*:=\{f^*=\widehat{f}\circ \bs^*: \widehat{f}\in \widehat{\sF}\},\\
	&\sE^*(f^*,g^*):=\widehat{\sE}(\widehat{f},\widehat{g}),\; f^*=\widehat{f}\circ \bs^*,  g^*=\widehat{g}\circ \bs^*\in \sF^*.
	\end{aligned}
\]
The result blow,  whose proof is trivial due to Lemma~\ref{LM32},  gives another regular representation of $(I,\fm,\sE,\sF)$.  

\begin{theorem}\label{THM43}
The quadratic form $(\sE^*,\sF^*)$ is a regular and irreducible Dirichlet form on $L^2(I^*,\fm^*)$.  Particularly,  $(I^*,\fm^*, \sE^*,\sF^*)$ is a regular representation of $(I,\fm,\sE,\sF)$.  
\end{theorem}

Denote by $X^*$ the Hunt process associated to $(\sE^*,\sF^*)$.  It is a regularized Markov process associated to $(I,\bs,\fm)$ as defined in \S\ref{SEC34}.  In addition,  the essentially exceptional set of $X^*$ is empty and its scale function is exactly $\bs^*$.

\subsection{Unregularized Markov process}\label{SEC9}

Let us turn to apply another ``inverse" map of $\bs$ to $\widehat{X}$.  For simplification we assume that $\bs$ is strictly increasing.  (Otherwise one may operate transformation of darning,  and then a similar discussion follows.) In this case (DK) trivially holds and (DM) reads as that $\fm$ is fully supported on $I$ and $\fm(\{x\})>0$ for $x\in D^0$; see Remark~\ref{RM22}.  In addition,  
\begin{equation}\label{eq:41}
\widehat{I}=\bs(I)\cup \{\bs(x-): x\in D^-\} \cup \{\bs(x+):x\in D^+\}. 
\end{equation}
Note incidentally that \eqref{eq:41} may not hold if $\bs$ is not strictly increasing.  

Recall that $(\sE,\sF)$ defined as \eqref{eq:25} turns to be a Dirichlet form on $L^2(I, \fm)$ in the wide sense in Theorem~\ref{LM12}.  The lemma below states more facts about it. 

%Now we come back to a triple $(I,\bs,\fm)$ defined in \S\ref{SEC2} but assume further that.  Then 

\begin{lemma}\label{LM81}
\begin{itemize}
\item[(1)] If $\bs$ is strictly increasing,  then $(\sE,\sF)$ is a Dirichlet form on $L^2(I,\fm)$.
\item[(2)] Assume that $\bs$ is strictly increasing.  Then $(\sE,\sF)$ is regular on $L^2(I,\fm)$ if and only if $\bs$ is continuous. 
\end{itemize}
\end{lemma}
\begin{proof}
\begin{itemize}
\item[(1)] Since $\bs$ is strictly increasing,  only the transformation of scale completion need be operated to obtain $(I^*,\bs^*,\fm^*)$ as in \S\ref{SEC41}.  Meanwhile $I^*$ can be treated as the union of $I$ and another set of at most countably many points,  $\fm^*|_I=\fm$ and $\fm^*(I^*\setminus I)=0$.  In view of the regularity of $(\sE^*,\sF^*)$ obtained in Theorem~\ref{THM6},  one can easily conclude the denseness of $\sF$ in $L^2(I,\fm)$.  
\item[(2)]  We first assert that for $\alpha,\beta\in (l,r)\setminus D$ with $\alpha<\beta$,  there is a constant $C_{\alpha,\beta}$ depending on $\alpha$ and $\beta$ such that
\begin{equation}\label{eq:81}
	\sup_{x\in [\alpha, \beta]}|f(x)|^2\leq C_{\alpha,\beta}\sE_1(f,f),\quad f\in \sF.
\end{equation}
In fact,  write $f=f^c+f^++f^-\in \sF$ and let $g^c=df^c/d\mu_c,  g^\pm=df^\pm/d\mu^\pm_d$. Take $x,y\in [\alpha, \beta]$ with $x<y$.  We have
\[
	|f^c(x)-f^c(y)|^2\leq \mu_c((x,y)) \int_x^y \left(g^c\right)^2d\mu_c \leq 2|\bs(y)-\bs(x)| \sE(f,f).   
\]
Analogous assertions hold for $f^\pm$.  Thus
\[
	|f(x)-f(y)|^2\leq 6|\bs(y)-\bs(x)|\sE(f,f).  
\]
It follows that for any $x,y\in [\alpha,\beta]$, 
\[
f(x)^2\leq 2f(y)^2+2|f(x)-f(y)|^2\leq 2f(y)^2 +12|\bs(\beta)-\bs(\alpha)|\sE(f,f),
\]
and hence for any $y\in [\alpha,\beta]$,
\[
\sup_{x\in [\alpha,\beta]}f(x)^2\leq 2f(y)^2 +12|\bs(\beta)-\bs(\alpha)|\sE(f,f).  
\]
Integrating both sides by $\fm$ on $[\alpha,\beta]$ we arrive at \eqref{eq:81}.  Now we prove the equivalence  between the regularity of $(\sE,\sF)$ and the continuity of $\bs$.  If $\bs$ is continuous,  then $(\sE,\sF)=(\sE^*,\sF^*)$,  $I=I^*$ and $\fm=\fm^*$.  In view of Theorem~\ref{THM6},  $(\sE,\sF)$ is regular on $L^2(I,\fm)$.  To the contrary,  argue by contradiction.  Assume that $\bs$ is not continuous at $x\in (l,r)$ while $(\sE,\sF)$ is regular.   Then we may take a function $f\in \sF$ such that $f$ is not continuous at $x$ and a sequence $f_n\in \sF\cap C_c(I)$ such that $\sE_1(f_n-f,f_n-f)\rightarrow 0$.   Take $\alpha,\beta\in (l,r)\setminus D$ such that $\alpha<x<\beta$.   Applying \eqref{eq:81} to $f_n-f$,  we get that $f_n$ converges to $f$ uniformly on $[\alpha,\beta]$.  Hence $f$ is continuous on $[\alpha,\beta]$,  as violates  the discontinuity of $f$ at $x$.  
\end{itemize}
That completes the proof. 
\end{proof}

Write the canonical regularized Markov process associated to $(I,\bs,\fm)$ as
\[
	\widehat X=\left\{\widehat{\Omega},\widehat{\sF}_t, \widehat{X}_t, (\widehat{\mathbf{P}}_{\widehat{x}})_{\widehat{x}\in \widehat{I}}, \widehat \zeta\right\},
\]
where $\widehat{\Omega}$ is the sample space,  $\widehat{\sF}_t$ is the filtration,  $\widehat{\mathbf{P}}_{\widehat{x}}$ is the probability measure on $\Omega$ with $\widehat{\mathbf{P}}_{\widehat{x}}(\widehat{X}_0=\widehat{x})=1$ and $\widehat{\zeta}$ is the lifetime of $\widehat{X}$.  Consider the ``inverse" map $\br: \widehat{I}\rightarrow I$:
\[
\br(\bs(x)):=x \text{ for }x\in I,\quad \br(\bs(x\pm)):=x \text{ for }x\in D^\pm.
\] 
In view of \eqref{eq:41},  $\br$ is well defined.  With $\widehat{X}$ and $\br$ at hand,  we proceed to put forward another simple Markov process on $I$.   Define $\dot \Omega:=\left\{\omega\in \widehat{\Omega}: \widehat{X}_0(\omega)\in \bs(I)\right\}\in \widehat\sF_0$ and 
\[
\begin{aligned}
	&\dot\sF_t:=\widehat{\sF}_t\cap \dot \Omega=\{A\cap \dot \Omega: A\in \widehat{\sF}_t\},\quad \dot X_t(\omega):=\br(\widehat{X}_t(\omega)), \; \omega\in \dot\Omega,    \\
	&\dot{\mathbf{P}}_x:=\widehat{\mathbf{P}}_{\bs(x)}|_{\dot \Omega},\; x\in I, \quad \dot \zeta(\omega):=\widehat{\zeta}(\omega),\; \omega\in \dot\Omega.  
\end{aligned}\]
Clearly $(\dot\sF_t)_{t\geq 0}$ is a right continuous filtration,  and $(\dot X_t)_{t\geq 0}$ is a family of random variables on $\dot \Omega$ adapted to $(\dot\sF_t)_{t\geq 0}$.  The following result is inspired by \cite[Theorem~3.6]{S79}. 

%Consider that $\bs$ is strictly increasing  but not continuous.  This lemma tells us that $(\sE,\sF)$ is not regular on $L^2(I,\fm)$.  
%In what follows we give an $\fm$-symmetric continuous simple Markov process $\dot X$ on $I$ whose Dirichlet form is $(\sE,\sF)$.  Here a simple Markov process means that $\dot X$ satisfies the Markov property but does not satisfy the strong Markov property.   Note that
%\begin{equation}\label{eq:80}
%	\widehat{I}=\bs(I)\cup \{\bs(x+): x\in D^+\}\cup \{\bs(x-):x\in D^-\},
%\end{equation}
%and the second and third sets on the right hand side are of zero $\widehat\fm$-measure.  Set a function $\mathbf{r}: \widehat{I}\rightarrow I$ by  

\begin{theorem}\label{THM82}
Assume that $\bs$ is strictly increasing but not continuous. 
The stochastic process 
\[
	\dot X=\{\dot \Omega,  \dot\sF_t, \dot X_t, (\dot{\mathbf{P}}_x)_{x\in I}, \dot\zeta\}
\]
is an $\fm$-symmetric continuous Markov process on $I$,  for which the strong Markov property fails.  Furthermore the Dirichlet form of $\dot X$ on $L^2(I,\fm)$ is $(\sE,\sF)$.  
\end{theorem}
\begin{proof}
Firstly we note that for any $A\in \dot \sF\subset \widehat{\sF}$,  $\dot{\mathbf{P}}_x(A)=\widehat{\mathbf{P}}_{\bs(x)}(A)$ is Borel measurable in $x$.  Secondly,  let us prove the continuity of all paths of $\dot X$.  The right continuity of all paths is due to that for $\widehat{X}$.  If $\widehat{X}_{t-}=\widehat{X}_t$,  then obviously $\dot{X}_{t-}=\dot{X}_t$.  If $\widehat{X}_{t-}\neq \widehat{X}_t$,  then the \emph{skip-free property} of $\widehat{X}$ (see,  e.g.,  \cite{K86,  L23b})  implies that $(\widehat X_{t-} \wedge \widehat X_t,  \widehat X_{t-}\vee \widehat X_t)=(\widehat{a}_k,\widehat{b}_k)$ for some interval $(\widehat{a}_k,\widehat{b}_k)$ in \eqref{eq:15}.  Note that $\br(\widehat{a}_k)=\br(\widehat{b}_k)$.  It follows that $\dot X_{t-}=\dot X_t$.  Hence $\dot X$ is a continuous stochastic process.

Next we turn to verify the Markov property of $\dot X$.  Denote by $\widehat{P}_t$ the transition functions of $\widehat{X}$,  i.e.  $\widehat{P}_t(\widehat{x},  \widehat{\Gamma})=\widehat{\mathbf{P}}_{\widehat{x}}(\widehat{X}_t\in \widehat{\Gamma})$ for $\widehat{x}\in \widehat{I}$ and $\widehat{\Gamma}\in \mathcal{B}(\widehat{I})$.  In view of \cite[Theorem~4.1]{L23b},  $\widehat{X}$ is equivalent to a \emph{quasidiffusion}.  Then $\widehat{P}_t$ admits a transition density with respect to $\widehat{\fm}$,  as mentioned in, e.g.,  \cite{K75}.  Particularly,  
\begin{equation}\label{eq:82}
	\widehat{P}_t(\widehat{x},  \widehat{I}\setminus \bs(I))=\widehat{\mathbf{P}}_{\widehat{x}}(\widehat{X}_t\in \widehat{I}\setminus \bs(I))=0. 
\end{equation}
Define $\dot P_t(x,\Gamma):=\dot{\mathbf{P}}_x(\dot{X}_t\in \Gamma)$ for $t\geq 0$, $x\in I$ and $\Gamma\in \mathcal{B}(I)$.  By the definition of $\dot X$,  we have $\dot P_t(x,\Gamma)=\widehat{\mathbf{P}}_{\bs(x)}(\br(\widehat{X}_t)\in \Gamma)=\widehat{\mathbf{P}}_{\bs(x)}(\widehat{X}_t\in \br^{-1}\Gamma)$. 
In view of \eqref{eq:82},  one gets 
\begin{equation}\label{eq:83}
	\dot P_t(x,\Gamma)=\widehat{\mathbf{P}}_{\bs(x)}\left(\widehat{X}_t\in \bs(\Gamma)\right)=\widehat{P}_t\left(\bs(x), \bs(\Gamma) \right).  
\end{equation}
To prove the Markov property of $\dot X$,  it suffices to show that for $\Gamma\in \mathcal{B}(I)$ and $A\in \dot\sF_t\subset \widehat{\sF}_t$,  
\begin{equation}\label{eq:84}
	\dot{\mathbf{P}}_x\left(\dot X_{t+s}\in \Gamma; A \right)=\int_A \dot P_s(\dot X_t,  \Gamma) d\dot{\mathbf{P}}_x.  
\end{equation}
In fact,  on account of the Markov property of $\widehat{X}$,   the left hand side of \eqref{eq:84} equals
\[
	\widehat{\mathbf{P}}_{\bs(x)}\left(\widehat{X}_{t+s}\in \br^{-1}\Gamma; A \right)=\int_{A} \widehat{P}_s(\widehat{X}_t, \br^{-1}\Gamma) d\widehat{\mathbf{P}}_{\bs(x)}.
\]
Using \eqref{eq:82} and \eqref{eq:83},  we get that $\dot{\mathbf{P}}_x(\dot X_{t+s}\in \Gamma; A)$ equals
\[
	\int_{A \cap \{\widehat{X}_t\in \bs(I)\}} \widehat{P}_s(\widehat{X}_t, \bs(\Gamma)) d\widehat{\mathbf{P}}_{\bs(x)}=\int_A \dot P_s(\dot X_t,  \Gamma) d\dot{\mathbf{P}}_x.  
\]
Hence \eqref{eq:84} is concluded. 

Fourthly we derive the symmetry of $\dot{X}$ with respect to $\fm$.  Note that \eqref{eq:83} gives the transition functions of $\dot X$.  Then the symmetry of $\dot X$ can be easily obtained by using the symmetry of $\widehat{X}$ with respect to $\widehat{\fm}$,  \eqref{eq:41} and $\widehat{\fm}=\fm\circ \bs^{-1}$.  In addition,  applying \cite[(1.3.17)]{FOT11} to $\dot P_t$ and noting \eqref{eq:35}, \eqref{eq:38-3},  one can verify that the Dirichlet form of $\dot X$ on $L^2(I,\fm)$ is $(\sE,\sF)$.  

Finally it suffices to show that the strong Markov property fails for $\dot X$.  If this is not true then $\dot X$ is a diffusion process on $I$ which is symmetric with respect to $\fm$.  By virtue of \cite[Theorem~3]{L21},  its Dirichlet form $(\sE,\sF)$ on $L^2(I,\fm)$ must be regular,  as violates the second assertion of Lemma~\ref{LM81} because $\bs$ is not continuous.  That completes the proof. 
\end{proof}
%\begin{remark}
%When $\bs$ is neither strictly increasing nor continuous,  we may operate only the transformation of darning on $I$,  as stated in \S\ref{SEC31},  to obtain a new ``interval" $I^*$.  Then $(\sE^*,\sF^*)$ defined as \eqref{eq:38-2} is a Dirichlet form on $L^2(I^*,\fm^*)$ but not regular.  Mimicking Theorem~\ref{THM82},  one may still find an $\fm^*$-symmetric continuous simple Markov process on $I^*$ whose Dirichlet form is $(\sE^*,\sF^*)$.  
%\end{remark}

\subsection{Ray-Knight compactification}\label{SEC43-2}

Finally we will show that the regularized Markov process $X^*$ obtained in Theorem~\ref{THM43} is the Ray-Knight compactification of the unregularized one $\dot{X}$.  Basic facts about Ray-Knight compactification are reviewed in Appendix~\ref{APPB}.  

Losing no generality but gaining much simplification,  we assume that $\bs$ is strictly increasing and $l\in I$,  i.e.  $l$ is reflecting.  The assumption $l\in I$ makes us concentrate on the right endpoint,  and general cases can be treated analogously.  Denote by $(\dot{R}_\alpha)_{\alpha>0}$ the Markov resolvent of $\dot{X}$ on $I_\partial:=I\cup \{\partial\}$,  where the trap $\partial$ is an isolated point attaching to $I$ if $r\in I$ and identified with $r$ if $r\notin I$.  

\subsubsection{Feller's boundary classification}

Let us first classify the endpoints $l$ and $r$ for $X^*$ in Feller's sense.  In view of Lemma~\ref{LM32},  we only need to state related terminologies by means of the parameters of $\widehat{X}$.  Put for  $\widehat{x}\in (\widehat{l},\widehat{r})$,  
\[
	{\widehat{\sigma}}(\widehat{x}):=\int_0^{\widehat{x}} {\widehat{\fm}}\left((0, {\widehat{y}}] \right) d{\widehat{y}}, \quad 
	{\widehat{\lambda}}({{\widehat{x}}}):=\int_{(0,  {\widehat{x}}]} {\widehat{y}} {\widehat{\fm}}(d{\widehat{y}}).
\]
For $\widehat{j}=\widehat{l}$ or $\widehat{r}$,  define ${\widehat{\sigma}}(\widehat{j}):=\lim_{{\widehat{x}}\rightarrow \widehat{j}}{\widehat{\sigma}}({\widehat{x}})$ and ${\widehat{\lambda}}(\widehat{j}):=\lim_{{\widehat{x}}\rightarrow \widehat{j}}{\widehat{\lambda}}({\widehat{x}})$. 
The following classification in Feller's sense is very well known.

\begin{definition}
The endpoint $\widehat{r}$ (resp.  $\widehat{l}$) for $\widehat X$ is called
\begin{itemize}
\item[(1)] \emph{regular},  if ${\widehat{\sigma}}({\widehat{r}})<\infty,  {\widehat{\lambda}}({\widehat{r}})<\infty$ (resp.  ${\widehat{\sigma}}({\widehat{l}})<\infty,  {\widehat{\lambda}}({\widehat{l}})<\infty$); 
\item[(2)] \emph{exit},  if ${\widehat{\sigma}}({\widehat{r}})<\infty,  {\widehat{\lambda}}({\widehat{r}})=\infty$ (resp.  ${\widehat{\sigma}}({\widehat{l}})<\infty,  {\widehat{\lambda}}({\widehat{l}})=\infty$); 
\item[(3)] \emph{entrance},  if ${\widehat{\sigma}}({\widehat{r}})=\infty,  {\widehat{\lambda}}({\widehat{r}})<\infty$ (resp.  ${\widehat{\sigma}}({\widehat{l}})=\infty,  {\widehat{\lambda}}({\widehat{l}})<\infty$); 
\item[(4)] \emph{natural},  if ${\widehat{\sigma}}({\widehat{r}})= {\widehat{\lambda}}({\widehat{r}})=\infty$ (resp.  ${\widehat{\sigma}}({\widehat{l}})= {\widehat{\lambda}}({\widehat{l}})=\infty$).  
\end{itemize}
Accordingly the endpoint $r$ (resp.  $l$) is called \emph{regular,  exit,  entrance} or \emph{natural}  for $X^*$ if so is $\widehat{r}$ (resp.  $\widehat{l}$) for $\widehat{X}$.  
\end{definition}
\begin{remark}
Note that $r$ is regular for $X^*$ if and only if $\bs(r)<\infty$ and $\fm((0,{r}))<\infty$.  Hence this regular property is identified with that in Definition~\ref{DEF21}.  If $r$ is exit,  then $\bs(r)<\infty$ and $\fm((0,r))=\infty$.  If $r$ is entrance,  then $\bs(r)=\infty$ and $\fm((0,r))<\infty$.  If $r$ is natural,  then $\bs(r)+\fm((0,r))=\infty$.  The left endpoint $l$ can be argued similarly. 
\end{remark}

The assumption $l\in I$ immediately implies that $l$ is regular in Feller's sense.  The lemma below due to \cite{L23c} is crucial to our investigation.  

\begin{lemma}\label{LM48}
Let $(\widehat{P}_t)_{t\geq 0}$ be the Markov semigroup of $\widehat{X}$.  The following hold:
\begin{itemize}
\item[\rm (1)]  If $r$ is reflecting or entrance,  then $(\widehat{P}_t)_{t\geq 0}$ acts  on $C(\overline{\bs([l,r])})$ as a Feller semigroup.  
\item[\rm (2)] If $r$ is absorbing,  exit or natural,  then $(\widehat{P}_t)_{t\geq 0}$ acts on $C_\infty(\widehat{I})$ as a Feller semigroup.  
\end{itemize}
\end{lemma}
\begin{remark}
Note that $\overline{\bs([l,r])}$ is compact,  and $\overline{\bs([l,r])}=\widehat{I}\cup \{\widehat{r}\}$.  The trap for $\widehat{X}$ is an isolated point attaching to $\overline{\bs([l,r])}$ in the first assertion,  while is identified with $\widehat{r}$ in the second assertion.  %When $r$ is natural,  $(\widehat{P}_t)_{t\geq 0}$ also acts as a Feller semigroup on $C(\overline{\bs([l,r])})$ but $\widehat{r}$ plays a similar role to the trap: Starting from $\widehat{r}$,  $\widehat{X}$ stays at $\widehat{r}$ forever; in addition,  starting from elsewhere,  $\widehat{X}$ cannot arrive at $\widehat{r}$.  
\end{remark}

\subsubsection{Reflecting and entrance cases}

We first treat the case that $r$ is reflecting or entrance.  Denote by $(R^*_\alpha)_{\alpha>0}$ the Markov resolvent of $X^*$ on $\bF:=I^*\cup \{r\}\cup \{\partial\}$,  where the trap $\partial$ is an isolated point attaching to $\bar{I}^*:=I^*\cup\{r\}$.  In view of Lemma~\ref{LM48},  $R^*_\alpha: C(\bar{I}^*)\rightarrow C(\bar{I}^*)$ is the resolvent of a Feller semigroup,  which is associated to a Feller process on $\bar{I}^*$.  When $r$ is entrance, i.e.  $r\in I^*$, this Feller process restricted to $I^*$ is identified with $X^*$,  while $X^*$ does not contain any information at $r$.  In abuse of notation we still denote by $X^*$ the Feller process on $\bF$ associated to $(R^*_\alpha)_{\alpha>0}$.  

\begin{theorem}\label{THM410}
Assume that $r$ is reflecting or entrance.  Then the Feller process $X^*$ on $\bF$ is the Ray-Knight compactification of $\dot X$.
\end{theorem}
\begin{proof}
Note that the first condition (i) in Lemma~\ref{LMB9} holds obviously for $E_\partial:=I_\partial$ and $\bF':=\bF$,  and in view of \eqref{eq:83},  we also get that (iv) in Lemma~\ref{LMB9} holds for $R_\alpha:=\dot R_\alpha$ and $\bar{R}'_\alpha:=R^*_\alpha$.  Take
\[
	G:=\{\dot{R}_\alpha f: \alpha>0,  f=f^*|_{E_\partial},  f^*\in C(\bF')\}.  
\]
Since $C(\bF')$ is separable and $\dot{R}_\alpha f=R^*_\alpha f^*|_{E_\partial}$,  one may easily check that $G$ satisfies the conditions listed before Lemma~\ref{LMB6}.  Hence Lemma~\ref{LMB6} leads to a Ray cone $S(G)$.  It suffices to prove that (i) and (ii) in Lemma~\ref{LMB9} hold for $S(G)$.  In fact,  since $\dot R_\alpha f$ extends to $R^*_\alpha f^*\in C(\bF')$ for every $f^*\in C(\bF')$,  it follows that $\{R^*_\alpha f^*: \alpha>0,  f^*\in C(\bF')\}\subset \overline{S(G)'}$.  Consequently $\overline{S(G)'}$ separates points of $\bF'$,  i.e.  (iii) in Lemma~\ref{LMB9} holds true.  On the other hand,  in view of the construction of $S(G)$ from $G$ (see,  e.g.,  \cite[Proposition~10.1]{Ge75}),  it is easy to verify that (ii) in Lemma~\ref{LMB9} also holds true.  That completes the proof. 
\end{proof}

\subsubsection{Absorbing,  exit and natural cases}

Regarding the case that $r$ is absorbing,  exit or natural,  we put $\bF:=I^*\cup \{\partial\}$ instead,  where the trap $\partial$ is identified with $r$ (in $\bar{I}^*$).  On account of Lemma~\ref{LM48},  $R^*_\alpha: C(\bF)\rightarrow C(\bF)$ is the Markov resolvent associated to the Feller process $X^*$ on $\bF$.  Mimicking the proof of Theorem~\ref{THM410},  we also have the following.

\begin{theorem}
Assume that $r$ is absorbing, exit or natural.  Then the Feller process $X^*$ on $\bF$ is the Ray-Knight compactification of $\dot X$.
\end{theorem}

%As stated in Corollary~\ref{COR310},  $X$ is transient if and only if either $-\infty<l\notin E$ or $\infty>r\notin E$.  Otherwise it is recurrent.  

\section{Examples}\label{SEC8}

In this section we present various classical Markov processes that can be treated as regularized or unregularized Markov processes associated to certain triple $(I,\bs,\fm)$.

\subsection{Snapping out diffusions}

Consider $l=-\infty,  r=\infty$, $\fm$ is the Lebesgue measure and
\[
\bs(x)=x,   \;x<0,\quad \bs(x)=x+2/\kappa,\; x\geq 0,
\]
where $\kappa>0$ is a given constant.  Clearly (DK) and \text{(DM)} hold true.  (Although $\bs$ is not continuous at $0$ as assumed in \S\ref{SEC2},  we may take another point in place of $0$.)

Denote by $H^1(J)$ the Sobolev space of order $1$ over an interval $J$.  Then it is easy to verify that $I=\bR$,  $I^*=(-\infty, 0-]\cup [0+,\infty)$ and 
\[
\begin{aligned}
	&\sF^*=\sF=\left\{f\in L^2(\bR): f|_{(0,\infty)} \in H^1((0,\infty)),  f|_{(-\infty, 0)}\in H^1((-\infty, 0))\right\},   \\
	&\sE^*(f,f)=\sE(f,f)=\frac{1}{2}\int_{\bR\setminus \{0\}}f'(x)^2dx+\frac{\kappa}{4}(f(0+)-f(0-))^2,\quad f\in \sF.  
\end{aligned}\]
Note that $(\sE,\sF)$ is a regular Dirichlet form on $L^2((-\infty, 0-]\cup [0+,\infty))$ and its associated Hunt process $X^*$, called \emph{snapping out Brownian motion} with parameter $\kappa$ (see \cite{L16} as well as \cite[\S3.4]{LS20} for its extension called \emph{snapping out diffusion process}),  is a regularized Markov  process associated to $(I,\bs,\fm)$.  

However $(\sE,\sF)$ is a Dirichlet form but not regular on $L^2(\bR)$, and Theorem~\ref{THM82} yields a continuous Markov process $\dot X$ whose Dirichlet form on $L^2(\bR)$ is $(\sE,\sF)$.  Before hitting $0$,  $\dot X$ moves as a Brownian motion,  and the excursions of $\dot X$ at $0$ can be described as follows: There exists a sequence of i.i.d.  random times $\{\tau_n: n\geq 1\}$ such that 
\[
S_n=\tau_0+\cdots+\tau_n\in Z:=\{t: \dot{X}_t=0\}, \quad \forall n\geq 0,
\]
where $\tau_0=\inf\{t>0: \dot X_t=0\}$,  and for $k\in \bN$,  
\begin{itemize}
\item[(1)] When $\dot X_0>0$,  the excursion intervals contained in $(S_{2k}, S_{2k+1})$ are on the right axis while those contained in $(S_{2k+1}, S_{2k+2})$ are on the left axis.
\item[(2)] When $\dot X_0<0$,  the excursion intervals contained in $(S_{2k}, S_{2k+1})$ are on the left axis while those contained in $(S_{2k+1}, S_{2k+2})$ are on the right axis.
\item[(3)] When $\dot X_0=0$,  the above two cases occur with equal probability.  
\end{itemize}
Particularly,  the strong Markov property at $\tau_0$ fails for $\dot X$.  
It is worth pointing out that the sequence of times $\{S_n: n\geq 1\}$ consists of the successive jumping times of snapping out Brownian motion,  and $\tau_n\overset{d}{=} \inf\{t>0: L_t>\xi\}$ is equal to the killing time of an \emph{elastic Brownian motion} in distribution,  where $L_t$ is the local time of a certain reflecting Brownian motion on $[0,\infty)$ at $0$ and $\xi$ is an independent exponential random variable with parameter $\kappa$.  %More details about elastic Brownian motion and snapping out Brownian motion are referred to in \cite{L16,  LS20}.  

\subsection{Random walk in one dimension}\label{SEC43}

Take $p,q\in \bN\cup \{\infty\}$ and a sequence of constants indexed by $\bZ_{-p,q}:=\{-p,1-p,\cdots,  -1, 0, 1,\cdots,  q\}\cap \mathbb{Z}$: $$c_{-p}<\cdots <c_{-1}<c_0<c_1<\cdots<c_q.$$  %Take another constant $c_{-\infty}<c_{-m}$ if $m\neq \infty$ and $c_\infty>c_n$ if $n\neq \infty$.  
Consider $l=-p,  r=q+1$ and
\begin{itemize}
\item[(a)] $\bs(x)=c_n$ for $x\in [n, n+1)$ and $n\in \mathbb{Z}_{-p,q}$.  %If $n\in \bN$ (resp.  $m\in \bN$),  define $\bs(x)=c_n$ for $x\in [n-1,\infty)$ (resp.  $\bs(x):=c_{-m}$ for $x\in (-\infty, -m)$).  
% further let $\bs(x)=c_{-\infty}$ for $x\in [-\infty, -m)$ if $m\neq \infty$, and let $\bs(x)=c_\infty$ for $x\in [n+1, \infty]$ if $n\neq \infty$;
\item[(b)] $\fm$ is fully supported, Radon on $I\cap \bR$ and $\fm(\{-\infty\})=\infty$ or $\fm(\{\infty\})=\infty$ whenever $p=\infty$ or $q=\infty$.  % When $n\in \bN$,  $\fm([n-1,\infty])<\infty$;  otherwise $\fm(\{\infty\})=\infty$.  When $m\in \bN$,  $\fm([-\infty, -m))<\infty$; otherwise $\fm(\{-\infty\})=\infty$.  

%is fully supported, Radon,  and absolutely continuous with respect to the Lebesgue measure and $\fm(\{-\infty\})=\fm(\{\infty\})=\infty$.  
\end{itemize}
Certainly (DK) and \text{(DM)} hold true.  

The triple $(I^*, \bs^*,\fm^*)$ obtained in \S\ref{SEC41} is as follows: $I^*$ is identified with the discrete space $\bZ_{-p,q}$,  and
\begin{itemize}
%\item[(a$^*$)] .   
%$$\pathbb{Z}_{-p,n}:=\{-p,1-p,\cdots,  -1, 0, 1,\cdots,  n\}\cap \mathbb{Z}.$$
%More precisely,  darning transformation collapses $[(k-1)+, k-]$ into an abstract point $p^*_k$ viewed as $k$ in $Z_{-m,n}$.  If $n\in \bN$,  then $[(n-1)+, \infty]$ is collapsed into $p^*_{n}$ viewed as $n$ in $Z_{-m,n}$.  If $m\in \bN$,  then $[-\infty, (-m)-]$ is collapsed into $p^*_{-m}$ viewed as $-m$ in $Z_{-m,n}$.  For convenience we write $I^*=\mathbb{Z}_{-m,n}$.  
\item[(a$^*$)] $\bs^*(n)=c_n$ for $n\in I^*$;
\item[(b$^*$)] $\fm^*(\{n\})=\fm([n, n+1))$ for $n\in \bZ$ with $-p\leq n\leq q-1$,  and $\fm^*(\{q\})=\fm([q, q+1])$ whenever $q\in \bN$.  %When $m\in \bN$,  $\fm^*(\{m\})=\fm([-\infty,  -m))$.  
\end{itemize}
%The right endpoint of $I^*$ can be classified as follows:
%\begin{itemize}
%\item[(i)] When $n$ is finite,  the right endpoint $n$ of $I^*$ must be reflecting;
%\item[(ii)] When $n=\infty$,  $\infty$ is approachable,  if and only if $\lim_{n\uparrow \infty}c_n<\infty$.  When $\infty$ is approachable,  $\infty$ is regular if and only if $\fm((0,\infty))<\infty$.  When $\infty$ is regular,  $\infty$ must be absorbing.  
%\end{itemize}
%The left endpoint can be argued analogously. 
Put
\[
\mu_{n,n+1}:=\frac{1}{2(c_{n+1}-c_n)},\; n\in \bZ\cap [-p, q-1],\quad \mu_{-p-1, -p}=\mu_{q,q+1}=0
\]
and 
\[
\mu_n:=\mu_{n-1,n}+\mu_{n,n+1},\quad n\in \bZ_{-p,q}.
\]
Then the Dirichlet form $(\sE^*,\sF^*)$ on $L^2(I^*,\fm^*)$ obtained in Theorem~\ref{THM43} is 
\begin{equation}\label{eq:51}
\begin{aligned}
&\sF^*=\{f^*\in L^2(I^*,\fm^*): \sE^*(f^*,f^*)<\infty, \\
&\qquad\qquad f^*(j^*)=0\text{ if } |j^*|=\infty, \lim_{n\rightarrow j^*}|c_n|<\infty\text{ for  }j^*=-p\text{ or }q \},\\
&\sE^*(f^*,f^*)=\sum_{-p\leq n\leq q-1} \mu_{n,n+1}\cdot (f^*(n+1)-f^*(n))^2,\quad f^*\in \sF^*. 
\end{aligned}\end{equation}
The regularized Markov process $X^*$ is a continuous time random walk on $I^*$.  More precisely,  for each $n\in I^*$,  $X^*$ waits at $n$ for a (mutually independent) exponential time with mean $\fm^*(\{n\})/\mu_n$ and then jumps to $n-1$ or $n+1$ according to the distribution
\[
	P_n(n-1)=\frac{\mu_{n-1,n}}{\mu_n},\quad P_n(n+1)=\frac{\mu_{n,n+1}}{\mu_n},\quad P_n(j)=0,\; j\neq n-1,  n+1. 
\]
When $\fm^*(\{n\})=\mu_n$,  $X^*$ is called \emph{in the constant speed}.  That means the holding exponential times are independent and identically distributed; see,  e.g.,  \cite[\S2.1]{K14}.  

Particularly,  when $p=q=\infty$, $c_n=n$ for $n\in \bZ$,  and $\fm|_\bR$ is the Lebesgue measure,  $X^*$ is identified with the \emph{continuous time simple random walk} on $\bZ$.  In addition,  the well-known \emph{birth and death processes} are also special examples of \eqref{eq:51}; see,  e.g.,  \cite{L22b}.

\subsection{Time-changed Brownian motions related to Fukushima subspaces}\label{SEC45}

Let $K\subset [0,1]$ be a generalized Cantor set (see, e.g.,  \cite[page 39]{F99}) of positive Lebesgue measure,  and write $[0,1]\setminus K$ as a disjoint union of open intervals:
\begin{equation}\label{eq:42}
	[0,1]\setminus K=\cup_{k\geq 1}(c_n,d_n).  
\end{equation}
Consider $l=0,r=1$ and
\begin{itemize}
\item[(a)]  $\bs(x)=x$ for $x\in K$ and $\bs(x):=c_n$ for $x\in (c_n,d_n)$ and $n\geq 1$;
\item[(b)] $\fm(dx)=1_K(x)dx$. % and $\delta_{c_n}, \delta_{d_n}$ are Dirac measures at $c_n,d_n$.  
\end{itemize}
It is easy to verify that (DK) and \text{(DM)} hold true. % only at each $d_n$

The transformation of scale completion in \S\ref{SEC41} divides each $d_n$ into $\{d_n-, d_n+\}$ and then the transformation of darning  collapses $[c_n, d_n-]$ into an abstract point $p^*_n$.  By regarding $p^*_n$ and $d_n+$ as $c_n$ and $d_n$ respectively,  one may identify $(I^*,\bs^*,\fm^*)$ with $(\widehat{I}, \widehat{\bs},  \widehat{\fm})$.  More precisely,
\[
	I^*=\widehat{I}=K,\quad \bs^*(x)=\widehat{\bs}(x)=x,\; x\in K,\quad \fm^*=\widehat{\fm}=\fm.  
\]
Let 
\begin{equation}\label{eq:43}
\sS^*:=\left\{f^*=h_*|_{K}: h_*\text{ is absolutely continuous on }(0,1)\text{ and } h'_{*} \in L^2((0,1))\right\}. 
\end{equation}
The Dirichlet form $(\sE^*,\sF^*)$ obtained in Theorem~\ref{THM43} is expressed as
\[
\begin{aligned}
&\sF^*=L^2(I^*,\fm^*)\cap \sS^*,\\
&\sE^*(f^*,f^*)=\frac{1}{2}\int_{K} \left(\frac{df^*}{dx}\right)^2dx+\frac{1}{2}\sum_{k\geq 1}\frac{\left(f^*(c_n)-f^*(d_n) \right)^2}{|d_n-c_n|},\quad f^*\in \sF^*.  
\end{aligned}
\]
Its associated Hunt process $X^*$ is a time-changed Brownian motion on $K$ with speed measure $\fm^*$.  This process has been utilized in \cite{LY17} to study the \emph{Fukushima subspaces} of 1-dimensional Brownian motion.  

%Let $\nu:=\sum_{k\geq 1}|d_n-c_n|\cdot \delta_{c_n}$,  and $F_\nu(x):=\nu((-\infty, x])$.  Denote by $\mathfrak{c}$ the standard Cantor function and by $\mu_\mathfrak{c}$ the Lebesgue-Stieltjes measure of $\mathfrak{c}$.  Consider
%\[
%\begin{aligned}
%	&I=[-\infty,  \infty],\quad m=\mu_\mathfrak{c},\\
%	&\bs=\mathfrak{c}+F_\nu.  
%\end{aligned}\]
%Clearly $\fm_\pm,  \fm_0$ have to be zero functions and (M0),  (M1) and (M2) are satisfied.  It is easy to verify that
%\[
%	I^*=K,  \quad \fm^*=\mu_\mathfrak{c}
%\]
%and 
%\[
%\begin{aligned}
%	&\sF^*= L^2(K,\mu_\mathfrak{c})\cap \sS^*, \\
%	&\sE^*(f,f)=\frac{1}{2}\int_K \left(\frac{df}{d\mu_\mathfrak{c}}\right)^2d\mu_\mathfrak{c}+\frac{1}{2}\sum_{k\geq 1}\frac{\left(f(c_n)-f(d_n)\right)^2}{|d_n-c_n|},
%\end{aligned}\]
%where 
%\[
%\begin{aligned}
%	\sS^*=\{f: \exists g\in L^2(K,\mu_\mathfrak{c})\text{ s.t. }& f(x)-f(0)=\int_0^xgd\mu_\mathfrak{c}+\sum_{k: a_k<x}(f(b_k)-f(a_k)) \\
%	&\text{for any }x\in K, \text{ and }\sE^*(f,f)<\infty\text{ with }df/d\mu_\mathfrak{c}:=g\}.  
%\end{aligned}\]
%The associated Markov process $X^*$ diffuses on $K$ and jumps only between $a_k$ and $b_k$.  

\subsection{Brownian motion on Cantor set}\label{SEC46}

Let $K\subset [0,1]$ be a generalized Cantor set as in \S\ref{SEC45} but we do not impose that it is of positive Lebesgue measure.  Set $K_m:=m+K=\{m+x: x\in K\}$ for $m\in \bZ$ and 
\[
\mathbf{K}:=\cup_{m\in \bZ}K_m.
\]
Write $[0,1]\setminus K$ as \eqref{eq:42} and let $c^m_n:=m+c_n,  d^m_n:=m+d_n$ for $m\in \bZ$.  Consider $l=-\infty,  r=\infty$ and
\begin{itemize}
\item[(a)] $\bs(x):=x$ for $x\in \mathbf{K}$ and $\bs(x):=c^m_n$ for $x\in (c^m_n,d^m_n)$ and $n\geq 1, m\in \bZ$; 
\item[(b)] $\fm(dx)=1_{\mathbf{K}}(x)dx+\sum_{n\geq 1,  m\in \bZ} |d^m_n-c^m_n|\cdot \left(\delta_{c^m_n}+\delta_{d^m_n}\right)/2$,  where $\delta_{c^m_n}, \delta_{d^m_n}$ are Dirac measures at $c^m_n,d^m_n$.  
\end{itemize}
Clearly (DK) and \text{(DM)} hold true.  %The measure reassignment is necessary only at each $d^n_n$,   and we take the special one that $\tm_+(d^n_n)=\frac{|d^n_n-c^n_n|}{2}$ for each $k\geq 1, n\in \bZ$ and $\tm_0,\tm_-\equiv 0$.  

Mimicking \S\ref{SEC45},  one yields 
\[
	I^*=\mathbf{K}, \quad \bs^*(x)=x,\; x\in \mathbf{K},\quad \fm^*=\fm.  
\]
Let $\sS^*$ be defined as \eqref{eq:43} with $\mathbf{K}$ and $\bR$ in place of $K$ and $(0,1)$ respectively.  Then the Dirichlet form $(\sE^*,\sF^*)$ on $L^2(I^*,\fm^*)$ is
\[
\begin{aligned}
&\sF^*=L^2(I^*,\fm^*)\cap \sS^*,\\
&\sE^*(f^*,f^*)=\frac{1}{2}\int_{\mathbf K} \left(\frac{df^*}{dx}\right)^2dx+\frac{1}{2}\sum_{n\geq 1,  m\in \bZ}\frac{\left(f^*(c^m_n)-f^*(d^m_n) \right)^2}{|d^m_n-c^m_n|},\quad f^*\in \sF^*.  
\end{aligned}
\]
Particularly,  when $K$ is of zero Lebesgue measure,  e.g.,  the standard Cantor set,  the strongly local part in the expression of $\sE^*(f^*,f^*)$ vanishes.  
The regularized Markov process $X^*$  is called a \emph{Brownian motion on Cantor set} in,  e.g.,  \cite{BEPP08}.  More precisely,  it is claimed in \cite{BEPP08} that $X^*$ is identified with the unique (in distribution) skip-free c\`adl\`ag process $\xi=(\xi_t)_{t\geq 0}$ (not necessarily a Markov process) on $\mathbf K$ such that both $\xi$ and $(\xi^2_t-t)_{t\geq 0}$ are martingales.  
%\begin{itemize}
%\item[(I)] $\xi$ is \emph{skip-free} in the sense that  for states $x<y<z$ in $\mathbf{K}$ and times $0\leq r<t<\infty$,  if either $\xi_r=x, \xi_t=z$ or $\xi_t=x, \xi_r=z$,  then $\xi_s=y$ for some time $s$ between $r$ and $t$;
%\item[(II)] $\xi$ is a martingale;
%\item[(III)] $(\xi^2_t-t)_{t\geq 0}$ is a martingale.  
%\end{itemize}
%We refer more details to \cite{BEPP08}.  

\appendix

%\section{Absolute continuity with respect to discontinuous scale}

\section{Proof of Lemma~\ref{LM39}}\label{APPA}

\begin{proof}[Proof of Lemma~\ref{LM39}]
%The first assertion is clear by the notes before this lemma.  We only prove \eqref{eq:33}.
Recall that $\bs(0)=0$.  
Set 
\[
\begin{aligned}
	\widehat{F}(\widehat{x})&:=|(0,\widehat x)\cap \widehat{I}|,\quad 0\leq \widehat{x}\leq  \widehat{r},\\
	\widehat{F}(\widehat{x})&:=-|(\widehat x,0)\cap \widehat{I}|,\quad \widehat{l}\leq \widehat{x}< 0,
\end{aligned}\]
where $|\cdot|$ stands for the Lebesgue measure.  Clearly $\widehat{F}$ is continuous and increasing,  $\widehat{F}(\bs(x))=\bs_c(x)$ for $x\in (l,r)$.  Set $\widehat{F}^{-1}(\widehat{t}):=\widehat{F}^{-1}(\{\widehat{t}\})=\{\widehat{x}: \widehat{F}(\widehat{x})=\widehat{t}\}$ for $\widehat{t}\in [\widehat{F}(\widehat{l}),\widehat{F}(\widehat{r})]$.  Note that except for at most countably many $\widehat{t}$,  $\widehat{F}^{-1}(\widehat{t})$ is a singleton.  Denote by this exceptional set by $\{\widehat{t}_p: p\geq 1\}$.  Particularly,  $\widehat{F}^{-1}(\widehat{t}_p)$ is an interval and 
\[
\widehat{H}:=\left(\cup_{p\geq 1} \widehat{F}^{-1}(\widehat{t}_p)\right)\cap \widehat{I}
\]
is of zero Lebesgue measure.  Set further 
\[
\widehat{G}(\widehat s):=\inf\{\widehat{x}\in [\widehat{l},\widehat{r}]:  \widehat{F}(\widehat{x})>\widehat s\}
\]
with the convention $\inf \emptyset=\infty$.  For $\widehat{t}\in  [\widehat{F}(\widehat{l}),\widehat{F}(\widehat{r}))\setminus \{\widehat{t}_p: p\geq 1\}$,  $\widehat{G}(\widehat{t})$ is the unique element in $\widehat{F}^{-1}(\widehat{t})$,  and in a little abuse of notation, we write $\widehat{G}(\widehat{t})=\widehat{F}^{-1}(\widehat{t})$.  Let
\[
	H:=\left\{x\in (l,r): \bs(x)\neq \widehat{G}(\bs_c(x))\right\}.  
\]
Then
\begin{equation}\label{eq:A1}
\bs(H):=\{\bs(x): x\in H\}=\{\widehat{x}\in \widehat{I}: \widehat{x}\neq \widehat{G}(\widehat{F}(\widehat{x}))\}=\widehat{H},
\end{equation} 
where the second and third identities hold up to a set of at most countably many points.  More precisely,  both $\bs(H)\setminus \{\widehat{x}\in \widehat{I}: \widehat{x}\neq \widehat{G}(\widehat{F}(\widehat{x}))\}$ and $\{\widehat{x}\in \widehat{I}: \widehat{x}\neq \widehat{G}(\widehat{F}(\widehat{x}))\}\setminus \bs(H)$ are subsets of $\{\widehat{l},\widehat{r}\}$,  $\{\widehat{x}\in \widehat{I}: \widehat{x}\neq \widehat{G}(\widehat{F}(\widehat{x}))\}\subset \widehat{H}$ and $\widehat{H}\setminus \{\widehat{x}\in \widehat{I}: \widehat{x}\neq \widehat{G}(\widehat{F}(\widehat{x}))\}$ contains at most countably many points.  We further assert
\begin{equation}\label{eq:A2}
\mu_c(H)=0.  
\end{equation}
To accomplish this,  note that for any $x\in H$,  $\bs_c(x)=\widehat{F}(\bs(x))\in \{\widehat{t}_p: p\geq 1\}$ and hence $|\bs_c(H)|=0$.  It follows that $\mu_c(H)\leq \mu_c\circ \bs_c^{-1}(\bs_c(H))=|\bs_c(H)|=0$,  where the image measure $\mu_c\circ \bs_c^{-1}$ of $\mu_c$ under $\bs_c$ is actually the Lebesgue measure;  see,  e.g.,  \cite[\S3.5, Exercise 36]{F99}.  

% and $\widehat{F}^{-1}(\bs_c(x))=\bs(x)$ for $x\in (l,r)\setminus D$.  Set further

Let $\widehat{f}=\widehat{h}|_{\widehat{I}}$ with $\widehat{h}\in \dot{H}^1_e\left((\widehat{l},\widehat{r})\right)$.  We are to prove that $f(\cdot):=\widehat{f}(\bs(\cdot))=\widehat{h}(\bs(\cdot))\in \sS$,  so that $\widehat{f}\in \widehat{\sS}$.  To do this, set
\[
\begin{aligned}
	&g^c(x):=\widehat{h}'(\widehat G(\bs_c(x))), \quad x\in (l,r);  \\
	&g^\pm(x):=\left\lbrace
	\begin{aligned}
	&\frac{\widehat{h}(\bs(x))-\widehat{h}(\bs(x\pm))}{\bs(x)-\bs(x\pm)},\quad x\in D^\pm, \\
	&0,\qquad\qquad\qquad\qquad\quad x\notin D^\pm. 
	\end{aligned}
	\right.
\end{aligned}\]
Since $\widehat{h}'\in L^2((\widehat{l},\widehat{r}))$ and $|\widehat{h}(\bs(x))-\widehat{h}(\bs(x\pm))|^2\leq \left|\int_{\bs(x)}^{\bs(x\pm)}\widehat{h}'(\widehat{t})^2d\widehat{t}\right| \cdot |\bs(x)-\bs(x\pm)|$,  it follows that $g^\pm\in L^2(I, \mu^\pm_d)$.  We assert that $g^c\in L^2(I, \mu_c)$.  In fact,  $\bs_c(l)=\widehat{F}(\widehat{l})$,  $\bs_c(r)=\widehat{F}(\widehat{r})$,  and using \cite[Chapter 0, Proposition~4.10]{RY99} with $A(\widehat t)=\widehat t$ and $u=\bs_c$ or $\widehat{F}$,  we get that $\int_l^r g^c(x)^2\mu_c(dx)$ is equal to 
\begin{equation}\label{eq:36}
	\int_{\bs_c(l)}^{\bs_c(r)}\widehat{h}'(\widehat{G}(\widehat t))^2d\widehat t=\int_{\widehat{l}}^{\widehat{r}} \widehat{h}' (\widehat{G}(\widehat{F}(\widehat{x})))^2d\widehat{F}(\widehat{x}) =\int_{\widehat{I}}  \widehat{h}' (\widehat{G}(\widehat{F}(\widehat{x})))^2d\widehat{x}.
\end{equation}
Since $\widehat{G}(\widehat{F}(\widehat{x}))=\widehat{x}$ for $\widehat{x}\in \widehat{I}\setminus \widehat{H}$ (see \eqref{eq:A1}) and $|\widehat{H}|=0$,  it follows that
\begin{equation}\label{eq:37}
	\int_l^r g^c(x)^2\mu_c(dx)=\int_{\widehat{I}}  \widehat{h}' (\widehat{x})^2d\widehat{x}<\infty.
\end{equation}
Hence $g^c\in L^2(I,\mu_c)$ is concluded.  
Now define for $x\in (l,r)$, 
\[
	f^+(x):=\int_{(0,x)}g^+(y)\mu^+_d(dy),\quad f^-(x):=\int_{(0,x]}g^-(y)\mu^-_d(dy).
\]
It suffices to show $f^c:=f-f^+-f^-\in \sS_c$.  Indeed,  take $0<x<r$ and we have
\begin{equation*}
\begin{aligned}
	f^c(x)-f^c(0)&=\widehat{f}(\bs(x))-\widehat{f}(0)-f^+(x)-f^-(x)=\int_{(0,\bs(x))\cap \widehat{I}} \widehat{h}'(\widehat{t})d\widehat{t}. 
\end{aligned}
\end{equation*}
Mimicking \eqref{eq:36} and \eqref{eq:37},  we get that
%\begin{equation}\label{eq:38-4}
%	f^c(x)-f^c(0)=\int_0^{\bs_c(x)}\widehat{h}'(\widehat{G}(\widehat{t}))d\widehat{t}.
%\end{equation}
%By virtue of \cite[Chapter 0, Proposition~4.10]{RY99} with $A(\widehat{t})=\widehat{t}$ and $u=\bs_c$,  it follows that
\[
f^c(x)-f^c(0)=\int_0^x g^c(y)\mu_c(dy).  
\]
On account of $g^c\in L^2(I,\mu_c)$,  we obtain that $f^c\in \sS_c$.  Therefore $f\in \sS$ is concluded. 

To the contrary,   take $f=f^c+f^++f^-\in \sS$ and let $\widehat{f}:=\widehat{\iota}f\in \widehat{\sS}$.  Define a function $\widehat{h}$ as follows:
\begin{equation}\label{eq:A5}
\begin{aligned}
	&\widehat{h}(\bs(x)):=\widehat{f}(\bs(x))=f(x),\quad x\in (l,r),  \\
	&\widehat{h}(\bs(x\pm)):=\widehat{f}(\bs(x\pm))=f(x\pm),\quad x\in D^\pm. 
\end{aligned}
\end{equation}
Note that $|\widehat{a}_k|+|\widehat{b}_k|<\infty$ due to $\widehat{I}\in \mathscr K$ (see \eqref{eq:15}).  
For $\widehat{x}\in (\widehat{a}_k,\widehat{b}_k),  k\geq 1$,  define further
\begin{equation}\label{eq:A6}
\widehat{h}(\widehat{x}):= \frac{\widehat{h}(\widehat b_k)-\widehat{h}(\widehat a_k)}{\widehat{b}_k-\widehat{a}_k}\cdot (\widehat{x}-\widehat{a}_k)+\widehat{h}(\widehat{a}_k). 
\end{equation}
We need to prove that $\widehat{h}\in \dot H^1_e((\widehat{l},\widehat{r}))$.  To do this,  define a function $\widehat{\varphi}$ as follows:
\begin{equation*}
	\widehat{\varphi}(\widehat{x}):=\left\lbrace 
	\begin{aligned}
	&\frac{df^c}{d\mu^c}(x),\quad \widehat{x}=\bs(x)\text{ with }x\in (l,r)\setminus H \text{ and the derivative exists}, \\
	&\frac{\widehat{h}(\widehat b_k)-\widehat{h}(\widehat a_k)}{\widehat{b}_k-\widehat{a}_k},\quad \widehat{x}\in (\widehat{a}_k,\widehat{b}_k), k\geq 1.  
	\end{aligned}
	\right. 
\end{equation*}
Note that $|\bs(H)|=|\widehat{H}|=0$ due to \eqref{eq:A1}.  
Let 
\[
	N:=\{x\in (l,r)\setminus H: df^c/d\mu^c(x)\text{ does not exist}\}, \quad \widehat{N}:=\bs(N).
	\]  
We point out that $|\widehat{N}|=0$,  so that $\widehat{\varphi}$ is defined a.e.  on $(\widehat{l},\widehat{r})$.  In fact,  $\widehat{N}\subset [\widehat{l},\widehat{r}]\setminus \widehat{H}$ and for $x\in N$,  $\bs(x)=\widehat{G}(\bs_c(x))=\widehat{F}^{-1}(\bs_c(x))$ because $\bs(x)\notin \widehat{H}$.  It follows that $\widehat{N}=\widehat{F}^{-1}(\bs_c(N))$, which is a subset of $\overline{\bs([l,r])}$,  and hence
\begin{equation}\label{eq:A8}
|\widehat N|=\int_{\widehat{F}^{-1}(\bs_c(N))}d\widehat{F}=|\bs_c(N)|=\mu_c\circ \bs_c^{-1}(\bs_c(N)),
\end{equation}
where the second and the third equalities hold because both $(d\widehat{F})\circ \widehat{F}^{-1}$ and $\mu_c\circ \bs^{-1}_c$,  i.e.  the image measures of $d\widehat{F}$ and $\mu_c$ under the maps $\widehat{F}$ and $\bs_c$ respectively,  are actually the Lebesgue measure.  Except for at most countably many points $\widehat t=\bs_c(x)\in \bs_c(N)$,  $\bs_c^{-1}(\{\widehat t\})$ is a singleton equalling $\{x\}$.   If is not a singleton,  $\bs_c^{-1}(\{\widehat t\})$ is an interval of zero $\mu_c$-measure.  Consequently,  \eqref{eq:A8} yields that $|\widehat{N}|=\mu_c(N)=0$.  
 Next,  we prove that $\widehat{\varphi}\in L^2((\widehat{l}, \widehat{r}))$.  
A straightforward computation yields that
\begin{equation}\label{eq:38}
\begin{aligned}
	\int_{\cup_{k\geq 1}(\widehat{a}_k,\widehat{b}_k)} \widehat{\varphi}(\widehat{x})^2d\widehat{x}&=\sum_{k\geq 1}\frac{\left(\widehat{h}(\widehat b_k)-\widehat{h}(\widehat a_k)\right)^2}{|\widehat{b}_k-\widehat{a}_k|} \\ 
	&=\sum_{x\in D^-}\frac{\left(f(x)-f(x-)\right)^2}{\bs(x)-\bs(x-)}+\sum_{x\in D^+}\frac{\left(f(x+)-f(x)\right)^2}{\bs(x+)-\bs(x)}  \\
	&=\int_I \left(\frac{df^-}{d\mu^-_d}\right)^2d\mu^-_d+\int_I \left(\frac{df^+}{d\mu^+_d}\right)^2d\mu^+_d<\infty.  
\end{aligned}
\end{equation}
Mimicking \eqref{eq:36} and \eqref{eq:37} and using \eqref{eq:A2}, we get that
\begin{equation}\label{eq:39-2}
\begin{aligned}
&\int_{(\widehat{l},\widehat{r})\setminus  \cup_{k\geq 1}(\widehat{a}_k,\widehat{b}_k)}   \widehat{\varphi}(\widehat{x})^2d\widehat{x}  \\  &\qquad = \int_l^r  \widehat{\varphi}(\widehat{G}(\bs_c(x)))^2d\bs_c(x) =\int_{I\setminus H} \widehat{\varphi}(\widehat{G}(\bs_c(x)))^2d\bs_c(x)\\
	&\qquad=\int_{I\setminus H} \widehat{\varphi}(\bs(x)))^2d\bs_c(x)  =\int_{I} \left(\frac{df^c}{d\mu^c}\right)^2d\mu^c<\infty. 
\end{aligned}\end{equation}
Hence $\widehat{\varphi}\in L^2((\widehat{l},\widehat{r}))$ is concluded.   Finally it suffices to show
\begin{equation}\label{eq:39}
	\widehat{h}(\widehat{x})-\widehat{h}(0)=\int_0^{\widehat{x}} \widehat{\varphi}(\widehat{t})d\widehat{t},\quad \widehat{x}\in (\widehat{l},\widehat{r}). 
\end{equation}
Consider first $\widehat{x}=\bs(x)$ for $x\in (0,r)\setminus D$.  The left hand side of \eqref{eq:39} is equal to
\begin{equation}\label{eq:314}
f(x)-f(0)=f^c(x)-f^c(0)+f^+(x)+f^-(x).  
\end{equation}
Note that 
\[
	f^\pm(x)=\pm\sum_{y\in (0,x)\cap D^\pm} (f(y\pm)-f(y))
\]
and 
\[
	\left(\cup_{y\in (0,x)\cap D^+} \left (\bs(y),\bs(y+)\right)\right) \cup \left(\cup_{y\in (0,x)\cap D^-} \left(\bs(y-),\bs(y)\right)\right) =\cup_{k: 0\leq  \widehat{a}_k<\widehat{b}_k\leq \bs(x)}(\widehat{a}_k,\widehat{b}_k).  
\]
It follows from \eqref{eq:A6} that
\begin{equation}\label{eq:315}
	f^+(x)+f^-(x)=\sum_{k: 0\leq \widehat{a}_k<\widehat{b}_k\leq \bs(x)}\int_{\widehat{a}_k}^{\widehat{b}_k} \widehat{\varphi}(\widehat{t})d\widehat{t}.  
\end{equation}
Mimicking \eqref{eq:36} and \eqref{eq:37},  we have
\[
	\int_{(0,\bs(x))\setminus \cup_{k\geq 1}(\widehat{a}_k,\widehat{b}_k)} \widehat{\varphi}(\widehat{t})d\widehat{t}=\int_{(0,x)}\widehat{\varphi}(\widehat{G}(\bs_c(y)))\mu_c(dy). 
\]
Since $\widehat{G}(\bs_c(y))=\bs(y)$ for $y\notin H$ and \eqref{eq:A2},    the last term is equal to
\begin{equation}\label{eq:316}
	\int_0^x \frac{df^c}{d\mu^c}d\mu^c=f^c(x)-f^c(0).  
\end{equation}
In view of \eqref{eq:314},  \eqref{eq:315} and \eqref{eq:316},  we obtain \eqref{eq:39} for $x\in (0,r)\setminus D$.  Analogously \eqref{eq:39} holds for $x\in (l,0)\setminus D$.  Due to the $\bs$-continuity of $f$,  $\widehat{h}$ is continuous on $(\widehat{l},\widehat{r})$.  By this continuity and $\widehat{\varphi}\in L^2((\widehat{l},\widehat{r}))$,  we can obtain \eqref{eq:39} for $\widehat{x}=\bs(x\pm)$ with $x \in D^\pm$.  Then the identity \eqref{eq:39} for $\widehat{x}\in \cup_{k\geq 1}(\widehat{a}_k,\widehat{b}_k)$ is obvious by means of \eqref{eq:A6}.  Eventually we conclude that $\widehat{h}\in \dot H^1_e((\widehat{l},\widehat{r}))$.  That completes the proof. 
\end{proof}

\section{Ray processes and Ray-Knight compactification}\label{APPB}

\subsection{Ray resolvent and Ray semigroup}

Let $\mathbf{F}$ be a compact metric space with Borel measurable $\sigma$-algebra $\mathcal{B}(F)$.  Further let $(U_\alpha)_{\alpha>0}$ be a \emph{Markov resolvent} on $\mathbf{F}$,  i.e.  for each $x\in \mathbf{F}$ and $\alpha>0$,  $\alpha U_\alpha(x,\cdot)$ is a probability measure on $\mathbf{F}$ and $(U_\alpha)_{\alpha>0}$ satisfies the resolvent equation
\[
	U_\alpha-U_\beta=(\beta-\alpha)U_\alpha U_\beta,\quad \forall \alpha, \beta>0.
\]
Set for $\alpha>0$,
\[
	\mathcal{S}_\alpha:=\{ f\in C(\mathbf{F}): f\geq 0, \beta U_{\alpha+\beta}f\leq f, \forall \beta\geq 0\},
\]
called the class of continuous \emph{$\alpha$-supermedian functions}.  

\begin{definition}
A Markov resolvent $(U_\alpha)_{\alpha>0}$ on $\mathbf{F}$ is called a \emph{Ray resolvent} if 
\begin{itemize}
\item[(i)] For each $\alpha>0$ and $f\in C(\mathbf{F})$,  $U_\alpha f\in C(\mathbf{F})$;
\item[(ii)] $\mathcal{S}_\infty:=\cup_{\alpha>0} \mathcal{S}_\alpha$ separates points of $\mathbf{F}$,  i.e.  for any $x,y\in \mathbf{F}$ with $x\neq y$,  there exists $f\in \mathcal{S}_\infty$ such that $f(x)\neq f(y)$.  
\end{itemize}
\end{definition}
\begin{remark}
The second condition in this definition can be replaced by a weaker one: $\mathcal{S}_1$ separates points of $\mathbf{F}$; see \cite[\S8.1]{CW05}.  
\end{remark}

The analytic part of Ray's theorem (see \cite{R59} and also \cite{RW00})  gives a \emph{Markov semigroup} $(P_t)_{t\geq 0}$ on $\mathbf{F}$,  i.e.  a family of kernels on $(F,\mathcal{B}(\mathbf{F}))$ such that for $x\in \mathbf{F}$ and $t\geq 0$,  $P_t(x,\cdot)$ is a probability measure on $\mathbf{F}$ and $P_{t+s}=P_tP_s$ for any $t,s\geq 0$.   Attention that the \emph{normal property},  i.e. $P_0(x,\cdot)=\delta_x$ for all $x\in \mathbf{F}$,  is not assumed for this definition of Markov semigroup! When $(P_t)_{t\geq 0}$ is further normal,  we call it a \emph{normal Markov semigroup}.   

\begin{theorem}[Ray]
Let $(U_\alpha)_{\alpha>0}$ be a Ray resolvent on $\mathbf{F}$.  Then there exists a unique Markov semigroup $(P_t)_{t\geq 0}$ such that
\begin{itemize}
\item[(i)] $t\mapsto P_tf(x)$ is right continuous on $[0,\infty)$ for each $f\in C(\mathbf{F})$ and $x\in \mathbf{F}$;
\item[(ii)] $U^\alpha f=\int_0^\infty e^{-\alpha t}P_tf dt$ for each $\alpha>0$ and $f\in C(\mathbf{F})$.   
\end{itemize} 
Furthermore 
\begin{equation}\label{eq:B1}
	D:=\{x\in \mathbf{F}: \alpha U^\alpha f(x)\rightarrow f(x)\;(\alpha\uparrow \infty), \forall f\in C(\mathbf{F})\}=\{x\in \mathbf{F}: P_0(x,\cdot)=\delta_x\}
\end{equation}
is Borel measurable.  
\end{theorem}

The Markov semigroup in this theorem is called the \emph{Ray semigroup}  of $(U_\alpha)_{\alpha>0}$ and the set $D$ defined as \eqref{eq:B1} is called the set of \emph{non-branching points}.  Accordingly $B:=\mathbf{F}\setminus D$ is called the set of \emph{branching points}.

It is well known that a Markov semigroup $(P_t)_{t\geq 0}$ is called a \emph{Feller semigroup} if it is normal and acts as a strongly continuous semigroup on $C(\mathbf{F})$.  Clearly a Feller semigroup is always a Ray semigroup.  The converse is not correct but we have the following result; see \cite[III.  (37.1)]{RW00}.  

\begin{proposition}
The Ray semigroup $(P_t)_{t\geq 0}$ is a Feller semigroup if and only if the set $B$ of branching points is empty.  
\end{proposition} 

\subsection{Ray process}

The probabilistic part of Ray's theorem is to introduce a Ray process associated to $(U_\alpha)_{\alpha>0}$.  We state a canonical method that can be found in \cite{Ge75}.  To do this,  set the family of sample paths
\[
\Omega:=\{\omega: [0,\infty)\rightarrow \mathbf{F}  \text{ is c\'adl\'ag and }\omega(t)\in D\text{ for any }t\geq 0\}
\]
and $X_t(\omega):=\omega(t)$ for $t\geq 0$.  Let $\mathcal{F}^0:=\sigma(X_t: t\geq 0)$ be the $\sigma$-algebra on $\Omega$ generated by $X:=(X_t)_{t\geq 0}$.  To make the definition of (strong) Markov process rigorous,  one need to introduce a so-called natural filtration as well as its augmentation.  For the sake of brevity we omit these details,  which are referred to in,  e.g.,  \cite{S88}.  

\begin{theorem}[Ray]
Let $(U_\alpha)_{\alpha>0}$ be a Ray resolvent on $\mathbf{F}$ and $(P_t)_{t\geq 0}$ be its Ray semigroup.  For any probability measure $\mu$ on $(\mathbf{F}, \mathcal{B}(\mathbf{F}))$,  there exists a unique probability measure $\mathbf{P}_\mu$ on $(\Omega, \mathcal{F}^0)$ under which $X$ is a c\'adl\'ag strong Markov process with transition semigroup $(P_t)_{t\geq 0}$ and initial distribution $\mu P_0$.  
\end{theorem}

The Markov process $X$ in this theorem is called a \emph{Ray process}.  Its initial distribution, not necessarily equalling $\mu$, is a probability measure supported on $D$,  i.e.  $\mu P_0(B)=0$.  In addition,  $X_t$ takes values in $D$ (while $X_{t-}$ may take values in $B$), and the restriction of $X$ to $D$ is a \emph{Borel right process}; see \cite[Chapter 1,  Theorem~9.13]{S88}.  For further properties of Ray processes we refer readers to \cite{Ge75,  S88,  CW05,  RW00}.  

\subsection{Ray-Knight compactification}

Ray processes are truly important because ``every Markov process is a Ray process at heart".  In his original paper \cite{R59},  Ray suggested the idea that every Markov process could be embedded in a Ray process,  although his proposed scheme had a subtle flaw.  Then Knight \cite{K65} patched it with a seemingly-innocuous lemma.  This is the reason why the name \emph{Ray-Knight compactification} appeared afterwards.    

For simplification let us begin with a locally compact space $E$ with countable basis.  Attach an additional point $\partial$ so that $E_\partial:= E\cup \{\partial\}$ becomes the Alexandroff compactification of $E$.  Denote by $\mathcal{B}_b(E_\partial)$ (resp.  $\mathcal{B}_b^+(E_\partial)$) the family of all bounded (resp.  positive bounded) Borel measurable functions on $E_\partial$.  For each $f\in \mathcal{B}_b(E_\partial)$ set $\|f\|_\infty:=\sup_{x\in E_\partial}|f(x)|$.  The ingredients for Ray-Knight compactification are the following two:
\begin{itemize}
\item[(1)] a Markov resolvent $(R_\alpha)_{\alpha>0}$ on $E_\partial$; 
\item[(2)] a family $G\subset \mathcal{S}_\infty\cap \mathcal{B}^+_b(E_\partial)$,  which separates $E_\partial$ and is separable with respect to $\|\cdot\|_\infty$.  
\end{itemize}
Here the separable property means that there exists a sequence $\{f_n:n\geq 1\}\subset G$ such that for any $f\in G$,  a certain subsequence of $\{f_n:n\geq 1\}$ converges to $f$ with respect to $\|\cdot\|_\infty$.  Note that $(R_\alpha)_{\alpha>0}$ can be easily obtained from a given Markov process on $E$ (the strong Markov property is not necessarily imposed).  The family $G$ is slightly difficult to put forward.  One typical way is to take a separable set $\mathcal{A}\subset \mathcal{B}^+_b(E_\partial)$ and then to verify $G:=\{R_\alpha f: f\in \mathcal{A}, \alpha>0\}$ separates $E_\partial$.  With $(R_\alpha)_{\alpha>0}$ and $G$ at hand,  Knight \cite{K65} proved the following lemma.

\begin{lemma}[Knight]\label{LMB6}
There exists a unique minimal convex cone $S(G)\subset \mathcal{B}^+_b(E_\partial)$ such that 
\begin{itemize}
\item[\rm (i)] $G\cup \{1_{E_\partial}\}\subset S(G)$ and $S(G)$ is separable with respect to $\|\cdot\|_\infty$;
\item[\rm (ii)] $R_\alpha f\in S(G)$ for any $\alpha>0$ and $f\in S(G)$;
\item[\rm (iii)] $f\wedge g\in S(G)$ for any $f,g\in S(G)$.  
\end{itemize}
\end{lemma}

The convex cone $S(G)$ is usually called a \emph{Ray cone}.  The Ray-Knight compactification is established in the following theorem.

\begin{theorem}[Ray-Knight compactification]\label{THMB7}
Let $(R_\alpha)_{\alpha>0}, G$ be given above and $S(G)$ be the Ray cone obtained in Lemma~\ref{LMB6}.  Then there exists a compact metric space $\mathbf{F}$ and a Ray resolvent $(\bar{R}_\alpha)_{\alpha>0}$ on $\bF$ such that
\begin{itemize}
\item[\rm (i)] $E_\partial$ is a Borel subset of $\bF$ and dense in $\bF$ with respect to the metric $\rho$ of $\bF$,  and $\mathcal{B}(E_\partial)=\{A\cap E_\partial: A\in \mathcal{B}(\bF)\}$;
\item[\rm (ii)] each $f\in S(G)$ can be extended to a function $\bar{f}\in C(\bF)$; 
\item[\rm (iii)] $\overline{S(G)}-\overline{S(G)}:=\{\bar{f}-\bar{g}: f,g\in S(G)\}$ is dense in $C(\bF)$;
\item[\rm (iv)] $\bar{R}_\alpha(x,A)=R_\alpha(x,A)$ for any $\alpha>0,  x\in E_\partial,  A\in \mathcal{B}(E_\partial)$.  
\end{itemize}
\end{theorem}
\begin{remark}\label{RMB8}
Here below is a standard argument to obtain $\bF$.  Take a sequence $\{h_n: n\geq 1\}\subset S(G)$ separating $E_\partial$ such that every $f\in S(G)$ is an accumulation point of $\{h_n: n\geq 1\}\subset S(G)$ with respect to $\|\cdot\|_\infty$.  Put 
\[
	\mathbf{W}:=\prod_{n\geq 1} [-\|h_n\|_\infty, \|h_n\|_\infty],
\]
which is a compact metric space endowed with certain metric,  e.g.,  $$\mathbf{d}(\xi,\eta):=\sum_{n\geq 1}2^{-n}\frac{|\xi_n-\eta_n|}{1+|\xi_n-\eta_n|},\quad \xi=(\xi_n)_{n\geq 1},  \eta=(\eta_n)_{n\geq 1}\in \mathbf{W}. $$
One may verify that 
\begin{equation}\label{eq:B2}
	\Phi: E\rightarrow \mathbf{W},\quad x\mapsto (h_n(x))_{n\geq 1}
\end{equation}
is a Borel measurable injection (see \cite[\S11.3]{Ge75}).  Let $\rho$ be the induced metric on $E$ with respect to $\mathbf{d}$,  i.e.  $\rho(x,y):=\mathbf{d}(\Phi(x),\Phi(y))$ for $x,y\in E$.  Eventually the completion of $E$ with respect to $\rho$ is the desirable compact metric space $\bF$.  This metric space depends on $G$ but not the choice of $\{h_n: n\geq 1\}$.  

The approach to $\bar{R}_\alpha$ is slightly complicated and 
the key step is to put 
\[
\bar{R}_\alpha \bar{f}:=\overline{R_\alpha f}\in C(\bF),\quad \forall f\in S(G).
\]
For more details we refer readers to,  e.g.,  \cite[Theorem~8.25]{CW05}.
\end{remark}

The compact metric space $\bF$ is called the Ray-Knight compactification of $E_\partial$ and $(\bar{R}_\alpha)_{\alpha>0}$ is called the Ray-Knight compactification of $(R_\alpha)_{\alpha>0}$.  When $(R_\alpha)_{\alpha>0}$ corresponds to a certain Markov process,  the Ray process corresponding to $(\bar{R}_\alpha)_{\alpha>0}$ is also called its Ray-Knight compactification.  We should emphasis that the relative topology of $E_\partial$ with respect to $\rho$ may be different from the original one.  In general they are not comparable. 

\subsection{Uniqueness of Ray-Knight compactification}

The Ray-Knight compactification for a Markov resolvent is not canonical because it depends on the choice of $G$ (or the Ray cone $S(G)$); see \cite[\S8.10]{CW05}.  However when $S(G)$ is fixed,  the conditions in Theorem~\ref{THMB7} determine the Ray-Knight compactification uniquely.  That is the following.

\begin{lemma}\label{LMB9}
Adopt the same notations of Theorem~\ref{THMB7}.  Let $\bF'$ be another compact metric space with the metric $\rho'$ and another Ray resolvent $(\bar{R}'_\alpha)_{\alpha>0}$ on $\bF'$ such that
\begin{itemize}
\item[\rm (i)] $E_\partial$ is a Borel subset of $\bF'$ with $\mathcal{B}(E_\partial)=\{A'\cap E_\partial: A'\in \mathcal{B}(\bF')\}$ and dense in $\bF'$ with respect to $\rho'$;
\item[\rm (ii)] each $f\in S(G)$ can be extended to a function $\bar{f}'\in C(\bF')$;
\item[\rm (iii)] $\overline{S(G)'}$ separates points of $\bF'$, where $\overline{S(G)'}:=\{\bar{f}': f\in S(G)\}$;
\item[\rm (iv)] $\bar{R}'_\alpha(x,A)=R_\alpha(x,A)$ for any $\alpha>0,  x\in E_\partial,  A\in \mathcal{B}(E_\partial)$.  
\end{itemize}
Then there exists a homoemorphism $j: \bF\rightarrow \bF'$ such that $j(x)=x$ for $x\in E_\partial$ and $\bar{R}'_\alpha(j(x),j(A))=\bar{R}_\alpha(x,A)$ for any $\alpha>0, x\in \bF$ and $A\in \mathcal{B}(\bF)$.  
\end{lemma}
\begin{proof}
Note that $\overline{S(G)'}$ is separable with respect to $\|\cdot\|_\infty$, where $\|\bar{f}'\|_\infty:=\|f\|_\infty$.  Using (iii) in this lemma one can take a sequence $\{h_n: n\geq 1\}\subset S(G)$ as in Remark~\ref{RMB8} with the property that $\{\bar{h}'_n:n\geq 1\}$ separates $\bF'$.  Define
\[
	\Phi':= \bF'\rightarrow \mathbf{W},  \quad x\mapsto (\bar{h}'_n(x))_{n\geq 1}.  
\]
Then \cite[Proposition~4.53]{F99} yields that $\Phi'$ is continuous and injective.  Since $\bF'$ is compact,  it follows from \cite[Proposition~4.28]{F99} that $\Phi':\bF'\rightarrow \mathbf{K}':=\{\Phi'(x):x\in \bF'\}\subset \mathbf{W}$ is a homeomorphism.   Note that the extension of \eqref{eq:B2} to $\bF$
\[
	\bar{\Phi}:\bF\rightarrow \mathbf{K}\subset \mathbf{W},
\]
where $\mathbf{K}$ is the closure of $\Phi(E_\partial)$ in $\mathbf{W}$,  is also a homeomorphism.  Since $\Phi'|_{E_\partial}=\Phi$ and $\mathbf{K}'$ is the closure of $\Phi'(E_\partial)=\Phi(E_\partial)$ in $\mathbf{W}$,  one gets that $\mathbf{K}=\mathbf{K}'$.  Consequently $j:=\Phi^{'-1}\circ \bar{\Phi}$ is a homeomorphism from $\bF$ to $\bF'$ such that $j(x)=x$ for any $x\in E_\partial$.  The relation between $\bar{R}_\alpha$ and $\bar{R}'_\alpha$ can be easily obtained by the fact that for any $f\in S(G)$, 
\[
	\bar{R}'_\alpha \bar{f}'= \overline{R_\alpha f}'= \overline{R_\alpha f} \circ j^{-1}=\left(\bar{R}_\alpha \bar{f}\right)\circ j^{-1}
\]
and $\bar{f}'=\bar{f}\circ j^{-1}$.  That completes the proof. 
\end{proof}

\section*{Acknowledgment}

The author would like to thank Professor Jiangang Ying for suggesting him to study the regular representations of $(\sE,\sF)$.  He would also like to thank Professor Patrick Fitzsimmons from University of California,  San Diego,  whose irradiative remarks lead him to a deep study of Ray-Knight compactification. 

%Here are further related references: \cite{LM20,  KW82,  S79,  K75,  W74}. 

\bibliographystyle{siam} % We choose the "plain" reference style
\bibliography{RegRep} % Entries are in the "refs.bib" file

\end{document}